\documentclass[review]{elsarticle}

\usepackage{lineno,hyperref,amsmath,amssymb,amsthm}

\numberwithin{equation}{section}
\newtheorem{theorem}{Theorem}[section]
\newtheorem{lemma}[theorem]{Lemma}

\newtheorem{remark}[theorem]{Remark}
\newtheorem{proposition}[theorem]{Proposition}

\begin{document}

\begin{frontmatter}

\title{Global well-posedness and optimal time-decay of 3D full compressible Navier-Stokes system}

\author{Wenwen Huo}
\ead{huowenwen2023@163.com}
\address{School of Mathematics, Harbin Institute of Technology, Harbin 150001, P.R. China}

\author{Chao Zhang} \cortext[Chao Zhang]{Corresponding author: Chao Zhang}
\ead{czhangmath@hit.edu.cn}
\address{School of Mathematics and Institute for Advanced Study in Mathematics, Harbin Institute of
Technology, Harbin 150001, P.R. China}

\begin{abstract}
In this paper, we investigate the global well-posedness and optimal time-decay of classical solutions for the 3-D full compressible Navier-Stokes system, which is given by the motion of the compressible viscous and heat-conductive gases. First of all, we study the global well-posedness of the Cauchy problem to the system when the initial data is small enough. Secondly, we show the optimal decay rates of the higher-order spatial derivatives of the $\dot{H}^{-s}$ $\left(0\leq s<\frac{3}{2}\right)$ negative Sobolev norms. Finally, under the assumption that the initial data is bounded in $L^{1}$-norm, we establish the upper and lower bounds of the optimal decay rates for the classical solutions.
\end{abstract}


\end{frontmatter}

\section{Introduction}\label{sec1}
In this paper, we consider the global existence and optimal large-time behavior of classical solutions to the following full compressible Navier-Stokes equations:
\begin{equation}\label{equ1-1}
\left\{
  \begin{array}{ll}
    \partial_{t}\rho+{\rm div}(\rho u)=0, \quad  x\in\mathbb{R}^{3},\,\,t>0,\\
    \partial_{t}(\rho u)+{\rm div}(\rho u\otimes u)+\nabla P(\rho,\theta)={\rm div}\mathbb{S}(u),\\
    \partial_{t}(\rho E)+{\rm div}(\rho Eu)+{\rm div}(Pu)={\rm div}(\mathbb{S}(u)u)+\tilde{\kappa}\Delta \theta,
     \end{array}
\right.
\end{equation}
where $\rho$, $u=(u_{1},u_{2},u_{3})$ and $\theta$ are the density, velocity and the absolute temperature of the fluid, and positive constant $\tilde{\kappa}$ is the coefficient of the heat conduction. The total energy $E$ and stress tensor $\mathbb{S}(u)$ are represented by
\begin{equation*}
E=e+\frac{|u|^{2}}{2}, \quad \mathbb{S}(u)=\tilde{\mu}(\nabla u+(\nabla u)^{T})+\tilde{\lambda}{\rm div} u\mathbb{I},
\end{equation*}
where $\mathbb{I}$ is the unit matrix, $\tilde{\mu}$ and $\tilde{\lambda}$ are viscosity coefficients satisfying
\begin{equation*}
\tilde{\mu}>0, \quad \tilde{\lambda}+\frac{2}{3}\tilde{\mu}\geq0.
\end{equation*}
The pressure $P$ and specific internal energy $e$ are given by
\begin{equation*}
P=R\rho \theta, \quad e=\frac{R}{\gamma-1}\theta,
\end{equation*}
where $R>0$ is a generic gas constant and $\gamma>1$ is the adiabatic constant.

Furthermore, we consider the problem of \eqref{equ1-1} with the initial condition
\begin{equation}\label{equ1-2}
(\rho,u,\theta)|_{t=0}=(\rho_{0},u_0,\theta_0)(x),
\end{equation}
and the far-field behavior
\begin{equation}\label{equ1-3}
\lim_{|x|\rightarrow \infty}(\rho,u,\theta)=(\rho_{\ast},0,\theta_{\ast}).
\end{equation}

In recent decades, there have been significant developments on the well-posedness and large-time behavior of solutions to the full compressible Navier-Stokes system. In the absence of vacuum, Matsumura and Nishida \cite{MN2} first established the global well-posedness for the system under the assumption of small initial data in $H^{3}$. When the initial data and the external force are sufficiently small, Matsumura and Nishida \cite{MN3} investigated the initial boundary value problems on the half-space and the exterior domain of any bounded region. If there is no external force present, the optimal $L^{2}$ time-decay rate for a strong solution to the system was first obtained by Matsumura and Nishida \cite{MN1} and the optimal $L^{p}$ $(p\geq2)$ time-decay rate was derived by Ponce \cite{Po}. Subsequently, if the external forces existed, Duan et al. \cite{DUYZ} considered the large time behavior for the strong solution provided that the initial perturbation is small in $H^{3}(\mathbb{R}^3)\cap L^{1}(\mathbb{R}^3)$. For the system with spherically symmetric solutions, Jiang \cite{Ji1} proved the global existence of classical solutions with (large) initial data, and Jiang \cite{Ji3} showed that the symmetric solutions decay to a constant state exponentially as time goes to infinity. Furthermore, the existence, asymptotic behavior of the solutions (weak, strong or classical) to the system has been obtained. We refer the readers to \cite{Ji2, Ji4, JZ, Ka, LL} for the  one-dimensional space.

When the initial vacuum is allowed, the authors in \cite{Li1, Li2, LX1} discussed the global well-posedness of strong solutions to the system for the one-dimensional space. In the three-dimensional case, Huang-Li \cite{HL1, HL2} obtained the existence of global classical and weak solutions for the system with large oscillations and established the blow-up criterion. When the domain is bounded, Wen-Zhu \cite{WZ1, WZ2} investigated the existence and uniqueness of global classical and strong solutions of the full compressible Navier-Stokes equations in one dimension and high dimension with large initial data. Later, Wen-Zhu \cite{WZ3} showed that the strong solution exists globally in time if the initial mass is small and obtained a surprisingly exponential decay rate of the strong solution in $\mathbb{R}^3$.

In addition, some important progress on the decay estimate has been achieved by a series of papers \cite{CLT, GWY, HHW, KS}. More precisely, for an exterior domain in $\mathbb{R}^{3}$, Kobayashi and Shibata \cite{KS} showed the $L_{q}$--$L_{p}$ estimates for solutions to the linearized equations and showed an optimal decay estimate for solutions to the nonlinear problem. Chen-Li-Tang \cite{CLT} established the lower bound of the decay rate to the linearized problem of the system. He-Huang-Wang \cite{HHW} obtained the uniform-in-time bound and large-time behavior of the system in the whole space. Besides, Gao-Wei-Yao \cite{GWY} proved the convergence of the global large solution to its associated constant equilibrium state with an explicit decay rate, and showed that the decay rate of the first order spatial derivative of large solution converging to zero in $L^{2}$-norm is $(t+1)^{-\frac{5}{4}}$.

In this paper, the first purpose is to investigate well-posedness of global classical solutions for the system under the assumptions of small initial data. Our second target is to study the optimal decay rate when the initial data $(\rho_{0}-\rho_{\ast},u_{0},\theta_{0}-\theta_{\ast})\in \dot{H}^{-s}(\mathbb{R}^3)$ for some $s\in[0,\frac{3}{2})$. If the initial data $(\rho_{0}-\rho_{\ast},u_{0},\theta_{0}-\theta_{\ast})$ belongs to $H^{N}\cap L^{1}$, the third purpose is to establish the upper and lower bounds of decay rate for the global solution itself and its spatial derivatives, and show that they are optimal.

Our main results are stated as follows.

\begin{theorem}\label{the1-1}
Assume that
\begin{equation*}
(\rho_{0}-\rho_{\ast},u_{0},\theta_{0}-\theta_{\ast})\in H^{N}(\mathbb{R}^3)
\end{equation*}
for an integer $N\geq3$, and there exists a small constant $\eta_{0}$ such that
\begin{equation}\label{equ1-4}
\left\|(\rho_{0}-\rho_{\ast},u_{0},\theta_{0}-\theta_{\ast})\right\|_{H^3}\leq \eta_{0}.
\end{equation}
Then the Cauchy problem \eqref{equ1-1}--\eqref{equ1-3} admits a unique solution $(\rho,u,\theta)$ satisfying
\begin{equation}\label{equ1-5}
\begin{split}
&\left\|(\rho-\rho_{\ast},u,\theta-\theta_{\ast})\right\|^{2}_{H^N}+
\int_0^t\left(\left\|\nabla \rho\right\|^{2}_{H^{N-1}}+\left\|\nabla u\right\|^{2}_{H^{N}}+\left\|\nabla \theta\right\|^{2}_{H^{N}}\right)d\tau\\
\leq& C\left\|(\rho_{0}-\rho_{\ast},u_{0},\theta_{0}-\theta_{\ast})\right\|^{2}_{H^N}
\end{split}
\end{equation}
for any $t>0$, where C is a positive constant independent of $\rho$, $u$ and $\theta$.
\end{theorem}

\begin{theorem}\label{the1-2}
Let all the assumptions of Theorem \ref{the1-1}  be satisfied and assume that $(\rho_{0}-\rho_{\ast},u_{0},\theta_{0}-\theta_{\ast})\in \dot{H}^{-s}(\mathbb{R}^3)$ for some $s\in[0,\frac{3}{2})$. Then
\begin{equation}\label{equ1-6}
\left\|\Lambda^{-s}(\rho-\rho_{\ast})\right\|^2_{L^2}+\left\|\Lambda^{-s}u\right\|^2_{L^2}+
\left\|\Lambda^{-s}(\theta-\theta_{\ast})\right\|^2_{L^2}\leq C,
\end{equation}
and
\begin{equation}\label{equ1-7}
\left\|\nabla^{l}(\rho-\rho_{\ast},u,\theta-\theta_{\ast})\right\|^2_{H^{N-l}}\leq C(t+1)^{-(l+s)}
\end{equation}
for $l=0,1,\cdot\cdot\cdot,N$.
\end{theorem}

\begin{theorem}\label{the1-3}
Let all the assumptions of Theorem \ref{the1-1} be satisfied and assume that
\begin{equation*}
(\rho_{0}-\rho_{\ast},u_{0},\theta_{0}-\theta_{\ast})\in H^{N}(\mathbb{R}^3)\cap L^{1}(\mathbb{R}^3)
\end{equation*}
for an integer $N\geq3$, and there exists a positive constant $A_{0}$ such that
\begin{equation}\label{equ1-8}
\left\|(\rho_{0}-\rho_{\ast},u_{0},\theta_{0}-\theta_{\ast})\right\|_{L^1}\leq A_{0}.
\end{equation}
Moreover, the Fourier transform $(\hat{n}_{0}(\xi),\hat{\omega}_{0}(\xi),\hat{\vartheta}_{0}(\xi))$ satisfies
\begin{equation}\label{equ1-9}
\inf_{|\xi|\ll 1}|\hat{n}_{0}(\xi)|\geq c_{0},
\inf_{|\xi|\ll 1}|\hat{\vartheta}_{0}(\xi)|\geq c_{0},
\sup_{|\xi|\ll 1}|\hat{\omega}_{0}(\xi)|\ll 1,
\sup_{|\xi|\ll 1}|(\hat{n}_{0}-\frac{\hat{\vartheta}_{0}}{\sqrt{\gamma-1}})(\xi)|\ll 1,
\end{equation}
where
\begin{equation*}
(n_{0},\omega_{0},\vartheta_{0}):=(\frac{\rho_{0}-\rho_{\ast}}{\rho_{\ast}},\frac{u_{0}}{\sqrt{R\theta_{\ast}}},
\frac{\theta_{0}-\theta_{\ast}}{\sqrt{\gamma-1}\theta_{\ast}}),
\end{equation*}
and $c_{0}$ is a positive constant. Then the solution $(\rho,u,\theta)$ has the following time-decay rate
\begin{equation}\label{equ1-10}
c_{\ast}(t+1)^{-\frac{3}{2}-k}\leq\left\|\nabla^{k}(\rho-\rho_{\ast},u,\theta-\theta_{\ast})\right\|^2_{L^{2}}\leq C(t+1)^{-\frac{3}{2}-k}
\end{equation}
for $0\leq k\leq N$.
\end{theorem}

\begin{remark} {\rm
Here are some comments concerning Theorems \ref{the1-1}--\ref{the1-3}.
\begin{itemize}
\item [(i)] By Lemma \ref{lem2-4},  we obtain that for $p\in(1,2]$, $L^{p}({\mathbb{R}^3})\subset\dot{H}^{-s}({\mathbb{R}^3})$ with $s=3(\frac{1}{p}-\frac{1}{2})\in[0,\frac{3}{2})$. Then by Theorem \ref{the1-2}, we have the following optimal decay results:
\begin{equation*}
\left\|\nabla^{l}(\rho,u,\theta)\right\|^2_{H^{N-l}}\leq C(t+1)^{-l-3(\frac{1}{p}-\frac{1}{2})}
\end{equation*}
for $l=0,1,\cdot\cdot\cdot,N$. Furthermore, the general optimal $L^{q}$ $(q>2)$ decay rates of the solution can be obtained by Gagliardo-Nirenberg inequality.
That is,
\begin{equation*}
\left\|\nabla^{l}(\rho,u,\theta)\right\|^2_{L^q}\leq\ C(t+1)^{-l-3(\frac{1}{p}-\frac{1}{q})}.
\end{equation*}

\item [(ii)] For the global well-posedness and the time-decay rates of the classical solution, we need the assumption for the smallness of the $H^{3}$-norm of
the initial data, while the higher-order Sobolev norms can be arbitrarily large. In addition, the decay results in Theorem \ref{the1-3} also need to assume that the $L^{1}$ norm of the initial date is bounded.

\item [(iii)] It is worth pointing out that the negative Sobolev estimate can only obtain the $N-th$ order spatial derivative of the system \eqref{equ1-1} with a decay rate
$(t+1)^{-(N-1+s)}$. However, we can generalize the results of the $N-th$ order spatial derivative to $(t+1)^{-(N+s)}$ by employing Sobolev inequality.

\item [(iv)] If the initial data $(\rho_{0}-\rho_{\ast},u_{0},\theta_{0}-\theta_{\ast})$ belongs to $H^{N}\cap L^{1}$, we can combine energy estimation with the decay rate of the linearized system to obtain the upper decay estimate at $k=0,1$, and then the decay rate of the $N-th$ order spatial derivative can be derived through mathematical induction. Finally, we can establish the lower decay estimate, which coincides with the upper one using Lemma \ref{lem2-5}.
\end{itemize}}
\end{remark}

The notations used in this article are as follows.
\small
\begin{itemize}
\item
$L^{p}({\mathbb{R}^3})$ denotes the usual Lebesgue spaces on $\mathbb{R}^3$, with norms
$$
\|u\|_{p}:=\left(\int_{\mathbb{R}^3}|u|^pdx\right)^{\frac{1}{p}}.
$$
\item
$H^{k}(\mathbb{R}^3)$ denotes the usual Sobolev spaces on $\mathbb{R}^3$, with norms
$$
\|u\|^{2}_{H^{k}}:=\int_{\mathbb{R}^3}\left(|u|^{2}+|\nabla u|^2+\cdots+|\nabla^{k}u|^{2}\right)dx.
$$
\item
$\dot{H}^{s}({\mathbb{R}^3})$ denotes the usual homogeneous Sobolev space on $\mathbb{R}^3$, with norms
$$
\left\| f\right\|_{\dot{H}^s}=\left\| \Lambda^{s}f\right\|_{L^2},
$$
where
$$
(\Lambda^{s}f)(x)=\frac{1}{(2\pi)^{3}}\int_{\mathbb{R}^3}|\xi|^{s}\hat{f}(\xi)e^{i\xi\cdot x}d\xi,
$$
and $\hat{f}(\xi)$ is the Fourier transform of $f$.
\item
$0\leq\varphi(\xi)\leq 1$ denotes a cut-off function such that
$$
\varphi(\xi)=
\left\{
  \begin{array}{ll}
    1,\quad \quad \text{for }  |\xi|\leq r_{0},\\
    0,\quad \quad \text{for }  |\xi|\geq R_{0},
     \end{array}
\right.
$$
where $\varphi(\xi)\in C^{\infty}(\mathbb{R}^3)$, constants $r_{0}$ and $R_{0}$ satisfying $0<r_{0}<R_{0}$. We define operators $P_{1}$ and $P_{\infty}$ on
$L^{2}$ by
$$
P_{1}f=\mathcal{F}^{-1}(\varphi(\xi)\hat{f}(\xi)),\quad \quad P_{\infty}f=\mathcal{F}^{-1}((1-\varphi(\xi))\hat{f}(\xi)),
$$
where $\mathcal{F}^{-1}(f)$ denotes the inverse Fourier transform of $f$.
\end{itemize}

The structure of this article is as follows. Section \ref{sec2} introduces some helpful lemmas. Section \ref{sec3} mainly focuses on deriving some energy estimates and using them to prove Theorem \ref{the1-1}. Section \ref{sec4} is dedicated to investigating the evolution of the negative Sobolev norm of the solution and proving Theorem \ref{the1-2}. Section \ref{sec5} is devoted to utilizing decay estimates for linearized problems to establish the upper and lower bounds of the decay rate (Theorem \ref{the1-3}).

\section{Preliminaries}\label{sec2}

In this section, we introduce some helpful results which will be used extensively in our paper.

\begin{lemma}
[Gagliardo-Nirenberg inequality, \cite{Ni}\label{lem2-1}] Let $l,s$ and $k$ be any real numbers satisfying $0\leq l,s<k$, and let $p,r,q\in[1,\infty]$ and $0\leq\theta\leq1$ such that
\begin{equation*}
\frac{l}{3}-\frac{1}{p}=\left(\frac{s}{3}-\frac{1}{r}\right)(1-\theta)+\left(\frac{k}{3}-\frac{1}{q}\right)\theta.
\end{equation*}
Then, for any $u\in W_{q}^{k}(\mathbb{R}^3)$, we have
\begin{equation*}
\left\|\nabla^{l}u\right\|_{L^p}\leq C\left\|\nabla^{s}u\right\|^{1-\theta}_{L^r}\left\|\nabla^{k}u\right\|^{\theta}_{L^q}.
\end{equation*}
Here, we require $0<\theta<1$ when $p=\infty$.
\end{lemma}

\begin{lemma}
[\label{lem2-2}Kato-Ponce inequality, \cite{KP}]
Let $s>0$ and $r\in(1,\infty)$. Then
\begin{equation*}
\left\|\nabla^{s}(fg)\right\|_{L^r}\leq C(\left\|g\right\|_{p_{1}}\left\|\nabla^{s}f\right\|_{q_{1}}+\left\|\nabla^{s}g\right\|_{p_{2}}\left\|f\right\|_{q_{2}}),
\end{equation*}
where $q_{1},q_{2}\in (1,\infty)$ and $\frac{1}{r}=\frac{1}{p_{1}}+\frac{1}{q_{1}}=\frac{1}{p_{2}}+\frac{1}{q_{2}}$.
\end{lemma}

\begin{lemma} [\label{lem2-3}\cite{PG}]
Let $f(\varphi)$ be smooth function of $\varphi$, with bounded derivatives of any order, and $\left\|\varphi\right\|_{L^{\infty}}\leq 1$. Then for any integer
$m\geq1$, we have
\begin{equation*}
\left\|\nabla^{m}f(\varphi)\right\|_{L^p}\leq C\left\|\nabla^{m}\varphi\right\|_{L^p}
\end{equation*}
for any $1\leq p\leq\infty$, where $C$ may depend on $f$ and $m$.
\end{lemma}

\begin{lemma} [\label{lem2-4}\cite{Zh}]
Let $0\leq s<\frac{3}{2}$, $1<p \leq 2$ and $\frac{1}{2}+\frac{s}{3}=\frac{1}{p}$, then
\begin{equation*}
\left\|f\right\|_{\dot{H}^{-s}}\leq C\left\|f\right\|_{L^p}.
\end{equation*}
\end{lemma}

\begin{lemma} [\label{lem2-5}\cite{Zh}]
Let $s,k\geq0$, and $l\geq0$, then
\begin{equation*}
\left\|\nabla^{l}f\right\|_{L^{2}}\leq C\left\|\nabla^{l+k}f\right\|^{1-\theta}_{L^2}\left\|f\right\|^{\theta}_{\dot{H}^{-s}}, \quad
\theta=\frac{k}{l+k+s}.
\end{equation*}
\end{lemma}

\section{Energy estimates}\label{sec3}

In this section, we first reformulate the problem \eqref{equ1-1}--\eqref{equ1-3}, and then establish the well-posedness of the classical solution by deriving a series of meaningful energy estimates.

To begin with, let us define
\begin{equation}\label{equ3-1}
n=\frac{\rho-\rho_{\ast}}{\rho_{\ast}}, \quad
\omega=\frac{u}{\sqrt{R\theta_{\ast}}}, \quad
\vartheta=\frac{\theta-\theta_{\ast}}{\sqrt{\gamma-1}\theta_{\ast}},
\end{equation}
and introduce the following constants as
\begin{equation}\label{equ3-2}
c=\sqrt{R\theta_{\ast}}, \quad
\sigma=\sqrt{(\gamma-1)R\theta_{\ast}}, \quad
\mu=\frac{\tilde{\mu}}{\rho_{\ast}}, \quad
\lambda=\frac{\tilde{\lambda}}{\rho_{\ast}}, \quad
\kappa=\frac{(\gamma-1)\tilde{\kappa}}{R\rho_{\ast}}.
\end{equation}
Then the system \eqref{equ1-1} is simplified as
\begin{equation}\label{equ3-3}
\left\{
  \begin{array}{ll}
    \partial_{t}n+c{\rm div}\omega=S_{1}, \quad  x\in\mathbb{R}^{3}, t>0,\\
    \partial_{t}\omega-\mu\Delta \omega-(\mu+\lambda)\nabla{\rm div}\omega+c\nabla n+\sigma\nabla\vartheta=S_{2},\\
    \partial_{t}\vartheta-\kappa\Delta\vartheta+\sigma{\rm div}\omega=S_{3},
     \end{array}
\right.
\end{equation}
with initial data
\begin{equation}\label{equ3-4}
(n,\omega,\vartheta)|_{t=0}=(n_{0},\omega_0,\vartheta_0)
=\left(\frac{\rho_{0}-\rho_{\ast}}{\rho_{\ast}}, \frac{u_{0}}{c}, \frac{\theta_{0}-\theta_{\ast}}{\sqrt{\gamma-1}\theta_{\ast}}\right),
\end{equation}
where $S_{i}$ $(i=1,2,3)$ are defined by
\begin{equation*}
\left\{
  \begin{array}{ll}
    S_{1}=-cn{\rm div}\omega-c\omega\cdot\nabla n,\\
    S_{2}=-c\omega\cdot\nabla \omega-h(n)(\mu\Delta \omega+(\mu+\lambda)\nabla{\rm div}\omega)+ch(n)\nabla n -\sigma g(n)\vartheta\nabla n,\\
    S_{3}=-c\omega\cdot\nabla\vartheta-\frac{\sigma^{2}}{c}\vartheta{\rm div}\omega+\frac{\sigma}{c}g(n)(2\mu|D\omega|^{2}
    +\lambda({\rm div}\omega)^{2})-\kappa h(n)\Delta\vartheta,
     \end{array}
\right.
\end{equation*}
and
\begin{equation*}
h(n)=\frac{n}{1+n},\quad g(n)=\frac{1}{1+n}, \quad D\omega=\frac{1}{2}(\nabla \omega+(\nabla \omega)^{T}).
\end{equation*}

In addition, we assume a priori that for sufficiently small positive constant $\eta_{1}$ satisfying
\begin{equation}\label{equ3-5}
\left\|n\right\|_{H^3}+\left\|\omega\right\|_{H^3}+\left\|\vartheta\right\|_{H^3}\leq\eta_{1}.
\end{equation}
Then we derived some energy estimates that play an important role in establishing the global existence of the solution.
\begin{lemma}\label{lem3-1}
Under the assumption \eqref{equ3-5}, it holds that
\begin{equation}\label{equ3-6}
\begin{split}
\frac{d}{dt}\left\|\nabla^{k}(n,\omega,\vartheta)\right\|^2_{L^2}
+\left\|\nabla^{k+1}(\omega,\vartheta)\right\|^2_{L^2}
\leq C\eta_{1}\left\|\nabla^{k+1}(n,\omega,\vartheta)\right\|^2_{L^2}
\end{split}
\end{equation}
for $k=0,\cdots,N$ ($N\geq3$).
\end{lemma}

\begin{proof}
Taking $\nabla^{k}$ to $\eqref{equ3-3}_{1}$, $\eqref{equ3-3}_{2}$, $\eqref{equ3-3}_{3}$, multiplying the resulting equation by $\nabla^{k}n$, $\nabla^{k}\omega$, $\nabla^{k}\vartheta$, and integrating over $\mathbb{R}^3$, we have
\begin{equation}\label{equ3-7}
\begin{split}
&\frac{1}{2}\frac{d}{dt}\left\|\nabla^{k}(n,\omega,\vartheta)\right\|^2_{L^2}
+\mu\left\|\nabla^{k+1}\omega\right\|^2_{L^2}+(\mu+\lambda)\left\|\nabla^{k}{\rm div}\omega\right\|^2_{L^2}+\kappa\left\|\nabla^{k+1}\vartheta\right\|^2_{L^2}\\
=&-c\int_{\mathbb{R}^3}\nabla^{k}(n{\rm div}\omega)\cdot\nabla^{k}ndx-c\int_{\mathbb{R}^3}\nabla^{k}(\omega\cdot\nabla n)\cdot\nabla^{k}ndx\\
&-c\int_{\mathbb{R}^3}\nabla^{k}(\omega\cdot\nabla \omega)\cdot\nabla^{k}\omega dx
-\int_{\mathbb{R}^3}\nabla^{k}\left[h(n)(\mu\Delta \omega+(\mu+\lambda)\nabla{\rm div}\omega)\right]\cdot\nabla^{k}\omega dx\\
&+c\int_{\mathbb{R}^3}\nabla^{k}(h(n)\nabla n)\cdot\nabla^{k}\omega dx
-\sigma\int_{\mathbb{R}^3}\nabla^{k}(g(n)\vartheta\nabla n)\cdot\nabla^{k}\omega dx\\
&-c\int_{\mathbb{R}^3}\nabla^{k}(\omega\cdot\nabla\vartheta)\cdot\nabla^{k}\vartheta dx
-\frac{\sigma^{2}}{c}\int_{\mathbb{R}^3}\nabla^{k}(\vartheta{\rm div}\omega)\cdot\nabla^{k}\vartheta dx\\
&+\frac{\sigma}{c}\int_{\mathbb{R}^3}\nabla^{k}\left[g(n)(2\mu|D \omega|^{2}+\lambda({\rm div}\omega)^{2})\right]\cdot\nabla^{k}\vartheta dx\\
&-\kappa\int_{\mathbb{R}^3}\nabla^{k}(h(n)\Delta\vartheta)\cdot\nabla^{k}\vartheta dx
:=\sum_{i=1}^{10}I_{i}.
\end{split}
\end{equation}

Combining H\"{o}lder's inequality, Kato-Ponce's inequality, and Sobolev embedding theorem, we estimate $I_{1}$ and $I_{2}$ as
\begin{equation*}
\begin{split}
|I_{1}+I_{2}|
\leq& C\left\|\nabla^{k}n\right\|_{L^6}\left\|\nabla^{k}{\rm div}(n\omega)\right\|_{L^{\frac{6}{5}}}\\
\leq& C\left\|\nabla^{k+1}n\right\|_{L^2}\left(\left\|\nabla^{k+1}n\right\|_{L^2}\left\| \omega\right\|_{L^3}+
\left\|n\right\|_{L^3}\left\|\nabla^{k+1}\omega\right\|_{L^2}\right)\\
\leq&C\eta_{1}\left(\left\|\nabla^{k+1}n\right\|^{2}_{L^2}+\left\|\nabla^{k+1}\omega\right\|^{2}_{L^2}\right).
\end{split}
\end{equation*}

For the term $I_{3}$, if $k=0$, we have
\begin{equation*}
\begin{split}
|I_{3}|\leq
C\left\|\omega\right\|_{L^3}\left\|\nabla \omega\right\|_{L^2}\left\|\omega\right\|_{L^6}
\leq C\left\|\omega\right\|_{L^3}\left\|\nabla \omega\right\|^{2}_{L^2}\leq C\eta_{1}\left\|\nabla \omega\right\|^{2}_{L^2}.
\end{split}
\end{equation*}
If $k\geq1$, by using Gagliardo-Nirenberg's inequality, it holds that
\begin{equation*}
\begin{split}
|I_{3}|\leq &C\left\|\nabla^{k}\omega\right\|_{L^6}\left\|\nabla^{k}(\omega\cdot\nabla \omega)\right\|_{L^{\frac{6}{5}}}\\
\leq& C\left\|\nabla^{k+1}\omega\right\|_{L^2}\left(\left\|\nabla^{k}\omega\right\|_{L^2}\left\|\nabla \omega\right\|_{L^3}+
\left\|\omega\right\|_{L^3}\left\|\nabla^{k+1}\omega\right\|_{L^2}\right)\\
\leq& C\left\|\nabla^{k+1}\omega\right\|_{L^2}\Big(\left\|\nabla^{k+1}\omega\right\|^{\frac{k}{k+1}}_{L^2}\left\| \omega\right\|^{\frac{1}{k+1}}_{L^2}\left\|\nabla^{\alpha} \omega\right\|^{\frac{k}{k+1}}_{L^2}\left\|\nabla^{k+1}\omega\right\|^{\frac{1}{k+1}}_{L^2}+
\left\|\omega\right\|_{L^3}\left\|\nabla^{k+1}\omega\right\|_{L^2}\Big)\\
\leq&C\eta_{1}\left\|\nabla^{k+1}\omega\right\|^{2}_{L^2},
\end{split}
\end{equation*}
where $\alpha$ is defined by
\begin{equation*}
\alpha=\frac{1}{2}+\frac{1}{2k}\in \left(\frac{1}{2},1\right].
\end{equation*}

For the term $I_{4}$, if $k=0$, we arrive at
\begin{equation*}
\begin{split}
|I_{4}|=&\left|\mu\int_{\mathbb{R}^3}\nabla \omega\nabla(h(n)\omega)dx+(\mu+\lambda)\int_{\mathbb{R}^3}{\rm div}\omega{\rm div}(h(n)\omega)dx\right|\\
\leq &C\left\|\nabla \omega\right\|_{L^2}(\left\|n\right\|_{L^{\infty}}\left\|\nabla \omega\right\|_{L^2}+
\left\|\nabla n\right\|_{L^{3}}\left\| \omega\right\|_{L^6})
\leq C\eta_{1}\left\| \nabla \omega\right\|^{2}_{L^2}.
\end{split}
\end{equation*}
On the other hand, if $k\geq1$, by H\"{o}lder's inequality, Gagliardo-Nirenberg's inequality and Lemma \ref{lem2-3}, we deduce that
\begin{equation*}
\begin{split}
|I_{4}|=&\left|\int_{\mathbb{R}^3}\nabla^{k-1}[h(n)(\mu\Delta \omega+(\mu+\lambda)\nabla{\rm div}\omega)]
\cdot{\rm div}\nabla^{k}\omega dx\right|\\
\leq& C\left\|\nabla^{k+1}\omega\right\|_{L^2}\left(\left\|\nabla^{k-1}n\right\|_{L^6}\left\|\nabla^{2} \omega\right\|_{L^3}+
\left\|n\right\|_{L^{\infty}}\left\|\nabla^{k+1}\omega\right\|_{L^2}\right)\\
\leq& C\left\|\nabla^{k+1}\omega\right\|_{L^2}\left(\left\|\nabla^{k+1}n\right\|^{\frac{k}{k+1}}_{L^2}
\left\| n\right\|^{\frac{1}{k+1}}_{L^2}\left\|\nabla^{\beta} \omega\right\|^{\frac{k}{k+1}}_{L^2}\left\|\nabla^{k+1}\omega\right\|^{\frac{1}{k+1}}_{L^2}
+\left\|n\right\|_{L^{\infty}}\left\|\nabla^{k+1}\omega\right\|_{L^2}\right)\\
\leq&C\eta_{1}\left(\left\|\nabla^{k+1}n\right\|^{2}_{L^2}+\left\|\nabla^{k+1}\omega\right\|^{2}_{L^2}\right),
\end{split}
\end{equation*}
where $\beta$ is defined by
\begin{equation*}
\beta=\frac{3}{2}+\frac{3}{2k}\in \left(\frac{3}{2},3\right].
\end{equation*}

For the term $I_{5}$, we only need to follow the idea of the estimation method of $I_{3}$ to obtain
\begin{equation*}
|I_{5}|
\leq C\eta_{1}\left(\left\|\nabla^{k+1}n\right\|^{2}_{L^2}+\left\|\nabla^{k+1}\omega\right\|^{2}_{L^2}\right).
\end{equation*}

Moreover, by using H\"{o}lder's inequality, Kato-Ponce's inequality, Young's inequality and Lemma \ref{lem2-3} again, the term $I_{6}$ can be bounded as
\begin{equation*}
\begin{split}
|I_{6}|\leq &C\left\|\nabla^{k}\omega\right\|_{L^6}\left\|\nabla^{k}(g(n)\vartheta\nabla n)\right\|_{L^{\frac{6}{5}}}\\
\leq& C\left\|\nabla^{k+1}\omega\right\|_{L^2}\left(\left\|\nabla^{k}g(n)\right\|_{L^6}\left\|\vartheta\nabla n\right\|_{L^{\frac{3}{2}}}+
\left\|g(n)\right\|_{L^6}\left\|\nabla^{k}(\vartheta\nabla n)\right\|_{L^{\frac{3}{2}}}\right)\\
\leq&C\eta_{1}\left(\left\|\nabla^{k+1}n\right\|^{2}_{L^2}+\left\|\nabla^{k+1}\omega\right\|^{2}_{L^2}+
\left\|\nabla^{k+1}\vartheta\right\|^{2}_{L^2}\right).
\end{split}
\end{equation*}

By following the idea used to $I_{3}$, it is easy to verify that
\begin{equation*}
\begin{split}
|I_{7}|+|I_{8}|\leq C\eta_{1}\left(\left\|\nabla^{k+1}\omega\right\|^{2}_{L^2}+\left\|\nabla^{k+1}\vartheta\right\|^{2}_{L^2}\right).
\end{split}
\end{equation*}

Next, also by H\"{o}lder's inequality, Kato-Ponce's inequality, Sobolev embedding theorem and Lemma \ref{lem2-3}, we estimate the terms $I_{9}$ as
\begin{equation*}
\begin{split}
|I_{9}|\leq &C\left\|\nabla^{k}\vartheta\right\|_{L^6}\left\|\nabla^{k}[g(n)(2\mu|D \omega|^{2}+\lambda({\rm div}\omega)^{2})]\right\|_{L^{\frac{6}{5}}}\\
\leq& C\left\|\nabla^{k+1}\vartheta\right\|_{L^2}\left(\left\|\nabla^{k}g(n)\right\|_{L^6}\left\|\nabla \omega\right\|^{2}_{L^{3}}+
\left\|g(n)\right\|_{L^3}\left\|\nabla^{k}|\nabla \omega|^{2}\right\|_{L^2}\right)\\
\leq& C\left\|\nabla^{k+1}\vartheta\right\|_{L^2}\left(\left\|\nabla^{k+1}n\right\|_{L^2}\left\| \nabla \omega\right\|^{2}_{L^3}
+\left\|g(n)\right\|_{L^3}\left\|\nabla \omega\right\|_{L^\infty}\left\|\nabla^{k+1}\omega\right\|_{L^2}\right)\\
\leq&C\eta_{1}\left(\left\|\nabla^{k+1}n\right\|^{2}_{L^2}+\left\|\nabla^{k+1}\omega\right\|^{2}_{L^2}+
\left\|\nabla^{k+1}\vartheta\right\|^{2}_{L^2}\right).
\end{split}
\end{equation*}

For the term $I_{10}$, we can follow the idea used to $I_{4}$ to derive
\begin{equation*}
\begin{split}
|I_{10}|\leq C\eta_{1}\left(\left\|\nabla^{k+1}n\right\|^{2}_{L^2}+\left\|\nabla^{k+1}\vartheta\right\|^{2}_{L^2}\right).
\end{split}
\end{equation*}

Summing up the estimates $I_{i}$ $(1\leq i\leq10)$, the estimate \eqref{equ3-6} follows.
\end{proof}

\begin{lemma}\label{lem3-2}
Under the assumption \eqref{equ3-5}, it holds that
\begin{equation}\label{equ3-8}
\begin{split}
\frac{d}{dt}\left\|\nabla^{k}(n,\omega,\vartheta)\right\|^2_{L^2}
+\left\|\nabla^{k+1}(\omega,\vartheta)\right\|^2_{L^2}
\leq C\eta_{1}\left(\left\|\nabla^{k}n\right\|^2_{L^2}+\left\|\nabla^{k+1}(\omega,\vartheta)\right\|^2_{L^2}\right)
\end{split}
\end{equation}
for $k=1,\cdots,N$ ($N\geq3$).
\end{lemma}
\begin{proof}
In this case, we have to estimate the terms containing $n$ in $I_{i}$ $(1\leq i\leq 10)$. First, through H\"{o}lder's inequality and Kato-Ponce's inequality, it is easy to verify that
\begin{equation*}
\begin{split}
|I_{1}|\leq &C\left\|\nabla^{k}n\right\|_{L^2}\left\|\nabla^{k}(n{\rm div}\omega)\right\|_{L^{2}}\\
\leq& C\left\|\nabla^{k}n\right\|_{L^2}\left(\left\|\nabla^{k}n\right\|_{L^2}\left\|\nabla \omega\right\|_{L^\infty}+
\left\|n\right\|_{L^\infty}\left\|\nabla^{k+1}\omega\right\|_{L^2}\right)\\
\leq&C\eta_{1}\left(\left\|\nabla^{k}n\right\|^{2}_{L^2}+\left\|\nabla^{k+1}\omega\right\|^{2}_{L^2}\right).
\end{split}
\end{equation*}

For the term $I_{2}$, integrating by parts, using H\"{o}lder's inequality, Kato-Ponce's inequality and Sobolev embedding theorem, we derive that
\begin{equation*}
\begin{split}
|I_{2}|\leq&C\left|\int_{\mathbb{R}^3}\omega\cdot\nabla^{k+1}n\cdot\nabla^{k}ndx\right|+
C\left|\int_{\mathbb{R}^3}\nabla^{k-1}(\nabla \omega\cdot\nabla n)\cdot\nabla^{k}ndx\right|\\
\leq &C\left|\int_{\mathbb{R}^3}\omega\cdot\nabla\frac{|\nabla^{k}n|^{2}}{2}dx\right|
+C\left\|\nabla^{k}n\right\|_{L^2}\left\|\nabla^{k-1}(\nabla \omega\cdot\nabla n)\right\|_{L^{2}}\\
\leq &C\left|\int_{\mathbb{R}^3}{\rm div}\omega\cdot\frac{|\nabla^{k}n|^{2}}{2}dx\right|
+C\left\|\nabla^{k}n\right\|_{L^2}\left\|\nabla^{k-1}(\nabla \omega\cdot\nabla n)\right\|_{L^{2}}\\
\leq&C\left\|\nabla \omega\right\|_{L^\infty}\left\|\nabla^{k}n\right\|^{2}_{L^2}
 +C\left\|\nabla^{k}n\right\|_{L^2}\left(\left\|\nabla^{k}\omega\right\|_{L^6}\left\|\nabla n\right\|_{L^3}+
\left\|\nabla \omega\right\|_{L^\infty}\left\|\nabla^{k}n\right\|_{L^2}\right)\\
\leq&C\eta_{1}\left(\left\|\nabla^{k}n\right\|^{2}_{L^2}+\left\|\nabla^{k+1}\omega\right\|^{2}_{L^2}\right).
\end{split}
\end{equation*}

Integrating by parts, then invoking H\"{o}lder's inequality, Kato-Ponce's inequality and Lemma \ref{lem2-3}, we estimate $I_{4}$ and $I_{5}$ as
\begin{equation*}
\begin{split}
|I_{4}|=&\left|\int_{\mathbb{R}^3}\nabla^{k-1}[h(n)(\mu\Delta \omega+(\mu+\lambda)\nabla{\rm div}\omega)]
\cdot{\rm div}\nabla^{k}\omega dx\right|\\
\leq& C\left\|\nabla^{k+1}\omega\right\|_{L^2}\left(\left\|\nabla^{k-1}n\right\|_{L^6}\left\|\nabla^{2} \omega\right\|_{L^3}+
\left\|n\right\|_{L^{\infty}}\left\|\nabla^{k+1}\omega\right\|_{L^2}\right)\\
\leq&C\eta_{1}\left(\left\|\nabla^{k}n\right\|^{2}_{L^2}+\left\|\nabla^{k+1}\omega\right\|^{2}_{L^2}\right),
\end{split}
\end{equation*}
and
\begin{equation*}
\begin{split}
|I_{5}|\leq&C\left|\int_{\mathbb{R}^3}\nabla^{k-1}(h(n)\nabla n)\cdot{\rm div}\nabla^{k}\omega dx\right|\\
\leq& C\left\|\nabla^{k+1}\omega\right\|_{L^2}\left(\left\|\nabla^{k-1}n\right\|_{L^6}\left\|\nabla n\right\|_{L^3}+
\left\|n\right\|_{L^\infty}\left\|\nabla^{k}n\right\|_{L^2}\right)\\
\leq&C\eta_{1}\left(\left\|\nabla^{k}n\right\|^{2}_{L^2}+\left\|\nabla^{k+1}\omega\right\|^{2}_{L^2}\right).
\end{split}
\end{equation*}

For the term $I_{6}$, if $k=1$, we have
\begin{equation*}
\begin{split}
|I_{6}|\leq&C\left|\int_{\mathbb{R}^3}\left(g(n)\vartheta\nabla n\right)\cdot{\rm div}\nabla\omega dx\right|\\
\leq &C\left\|g(n)\right\|_{L^\infty}\left\|\vartheta\right\|_{L^\infty}
\left\|\nabla n\right\|_{L^2}\left\|\nabla^{2}\omega\right\|_{L^2}\\
\leq &C\eta_{1}\left(\left\|\nabla n\right\|^{2}_{L^2}+\left\|\nabla^{2}\omega\right\|^{2}_{L^2}\right).
\end{split}
\end{equation*}
On the other hand, if $k\geq2$, employing H\"{o}lder's inequality, Kato-Ponce's inequality, Gagliardo-Nirenberg's inequality and
Lemma \ref{lem2-3}, we arrive at
\begin{equation*}
\begin{split}
|I_{6}|\leq&C\left|\int_{\mathbb{R}^3}\nabla^{k-1}(g(n)\vartheta\nabla n)\cdot{\rm div}\nabla^{k}\omega dx\right|\\
\leq& C\left\|\nabla^{k+1}\omega\right\|_{L^2}\Big[\left\|\nabla^{k-1}g(n)\right\|_{L^6}\left\|\vartheta\right\|_{L^{6}}
\left\|\nabla n\right\|_{L^{6}}\\&+
\left\|g(n)\right\|_{L^\infty}\left(\left\|\vartheta\right\|_{L^\infty}\left\|\nabla^{k}n\right\|_{L^2}+
\left\|\nabla^{k-1}\vartheta\right\|_{L^6}\left\|\nabla n\right\|_{L^3}\right)\Big]\\
\leq& C\left\|\nabla^{k+1}\omega\right\|_{L^2}\Big(\left\|\nabla^{k}n\right\|_{L^2}\left\|\vartheta\right\|_{L^{6}}
\left\|\nabla n\right\|_{L^{6}}+\left\|\vartheta\right\|_{L^\infty}\left\|\nabla^{k}n\right\|_{L^2}\\
&+\left\|\nabla\vartheta\right\|^{\frac{1}{k}}_{L^2}\left\| \nabla^{k+1}\vartheta\right\|^{\frac{k-1}{k}}_{L^2}
\left\|\nabla^{\sigma} n\right\|^{\frac{k-1}{k}}_{L^2}\left\|\nabla^{k}n\right\|^{\frac{1}{k}}_{L^2}\Big)
\\
\leq&C\eta_{1}\left(\left\|\nabla^{k}n\right\|^{2}_{L^2}+\left\|\nabla^{k+1}\omega\right\|^{2}_{L^2}+
\left\|\nabla^{k+1}\vartheta\right\|^{2}_{L^2}\right),
\end{split}
\end{equation*}
where $\sigma$ is defined by
\begin{equation*}
\sigma=\frac{k}{2(k-1)}\in\left[1,\frac{1}{2}\right).
\end{equation*}

By using H\"{o}lder's inequality, Kato-Ponce's inequality, Sobolev embedding theorem, and Lemma \ref{lem2-3}, we deduce that
\begin{equation*}
\begin{split}
|I_{9}|\leq &C\left\|\nabla^{k}\vartheta\right\|_{L^6}\left\|\nabla^{k}[g(n)(2\mu|D \omega|^{2}+\lambda({\rm div}\omega)^{2})]\right\|_{L^{\frac{6}{5}}}\\
\leq& C\left\|\nabla^{k+1}\vartheta\right\|_{L^2}\left(\left\|\nabla^{k}g(n)\right\|_{L^2}\left\|\nabla \omega\right\|^{2}_{L^{6}}+
\left\|g(n)\right\|_{L^3}\left\|\nabla^{k}|\nabla \omega|^{2}\right\|_{L^2}\right)\\
\leq&C\eta_{1}\left(\left\|\nabla^{k}n\right\|^{2}_{L^2}+\left\|\nabla^{k+1}\omega\right\|^{2}_{L^2}+
\left\|\nabla^{k+1}\vartheta\right\|^{2}_{L^2}\right).
\end{split}
\end{equation*}

Moreover, by following the idea used to $I_{4}$, the term $I_{10}$ can be bounded as
\begin{equation*}
|I_{10}|\leq C\eta_{1}\left(\left\|\nabla^{k}n\right\|^{2}_{L^2}+\left\|\nabla^{k+1}\vartheta\right\|^{2}_{L^2}\right).
\end{equation*}

Summing up the estimates $I_{i}$ $(1\leq i\leq10)$, the estimate \eqref{equ3-8} follows.
\end{proof}
\begin{lemma}\label{lem3-3}
Under the assumption \eqref{equ3-5}, it holds that
\begin{equation}\label{equ3-9}
\begin{split}
&\frac{d}{dt}\int_{\mathbb{R}^3}\nabla^{k}\omega\cdot\nabla^{k+1}n dx
+\frac{c}{2}\left\|\nabla^{k+1}n\right\|^2_{L^2}-c\left\|\nabla^{k}{\rm div}\omega\right\|^2_{L^2}\\
\leq&C\eta_{1}\left(\left\|\nabla^{k+1}(n,\omega,\vartheta)\right\|^2_{L^2}
+\left\|\nabla^{k+2}\omega\right\|^2_{L^2}\right)
+C\left(\left\|\nabla^{k+1}\vartheta\right\|^2_{L^2}+\left\|\nabla^{k+2}\omega\right\|^2_{L^2}\right)
\end{split}
\end{equation}
for $k=0,\cdots,N-1$ ($N\geq3$).
\end{lemma}
\begin{proof}
Taking $\nabla^{k+1}$ to $\eqref{equ3-3}_{1}$, $\nabla^{k}$ to $\eqref{equ3-3}_{2}$, multiplying the resulting equation by $\nabla^{k}\omega$, $\nabla^{k+1}n$, and integrating over $\mathbb{R}^3$, we have
\begin{equation}\label{equ3-10}
\begin{split}
&\frac{d}{dt}\int_{\mathbb{R}^3}\nabla^{k}\omega\cdot\nabla^{k+1}n dx
+c\left\|\nabla^{k+1}n\right\|^2_{L^2}
-c\left\|\nabla^{k}{\rm div}\omega\right\|^2_{L^2}\\
=&\int_{\mathbb{R}^3}\nabla^{k}(\mu\Delta \omega+(\mu+\lambda)\nabla{\rm div}\omega)\cdot\nabla^{k+1}ndx
-\sigma\int_{\mathbb{R}^3}\nabla^{k+1}\vartheta\cdot\nabla^{k+1}ndx\\
&+c\int_{\mathbb{R}^3}\nabla^{k}{\rm div}(n\omega)\cdot{\rm div}\nabla^{k}\omega dx
-c\int_{\mathbb{R}^3}\nabla^{k}(\omega\cdot\nabla \omega)\cdot\nabla^{k+1}ndx\\
&-\int_{\mathbb{R}^3}\nabla^{k}\left[h(n)(\mu\Delta \omega+(\mu+\lambda)\nabla{\rm div}\omega)\right]\cdot\nabla^{k+1}ndx\\
&+c\int_{\mathbb{R}^3}\nabla^{k}(h(n)\nabla n)\cdot\nabla^{k+1}ndx
-\sigma\int_{\mathbb{R}^3}\nabla^{k}(g(n)\vartheta\nabla n)\cdot\nabla^{k+1}ndx\\
:=&\sum_{i=11}^{17}I_{i}.
\end{split}
\end{equation}

First, H\"{o}lder's inequality and Young's inequality imply that
\begin{equation*}
\begin{split}
|I_{11}|
\leq &C\left\|\nabla^{k+2}\omega\right\|_{L^2}\left\|\nabla^{k+1}n\right\|_{L^{2}}
\leq C\left\|\nabla^{k+2}\omega\right\|^{2}_{L^2}+\frac{c}{4}\left\|\nabla^{k+1}n\right\|^{2}_{L^{2}},
\end{split}
\end{equation*}
and
\begin{equation*}
\begin{split}
|I_{12}|
\leq &C\left\|\nabla^{k+1}\vartheta\right\|_{L^2}\left\|\nabla^{k+1}n\right\|_{L^{2}}
\leq C\left\|\nabla^{k+1}\vartheta\right\|^{2}_{L^2}+\frac{c}{4}\left\|\nabla^{k+1}n\right\|^{2}_{L^{2}}.
\end{split}
\end{equation*}

For $I_{13}$--$I_{16}$, by H\"{o}lder's inequality, Kato-Ponce's inequality, and Sobolev embedding theorem, we obtain
\begin{equation*}
\begin{split}
|I_{13}|
\leq &C\left\|{\rm div}\nabla^{k}\omega\right\|_{L^2}\left\|\nabla^{k}{\rm div}(n\omega)\right\|_{L^{2}}\\
\leq& C\left\|\nabla^{k+1}\omega\right\|_{L^2}\left(\left\|\nabla^{k+1}n\right\|_{L^2}\left\|\omega\right\|_{L^\infty}+
\left\|n\right\|_{L^{\infty}}\left\|\nabla^{k+1}\omega\right\|_{L^2}\right)\\
\leq&C\eta_{1}\left(\left\|\nabla^{k+1}n\right\|^{2}_{L^2}+\left\|\nabla^{k+1}\omega\right\|^{2}_{L^2}\right),
\end{split}
\end{equation*}
\begin{equation*}
\begin{split}
|I_{14}|
\leq &C\left\|\nabla^{k+1}n\right\|_{L^2}\left\|\nabla^{k}(\omega\cdot\nabla \omega)\right\|_{L^{2}}\\
\leq& C\left\|\nabla^{k+1}n\right\|_{L^2}\left(\left\|\nabla^{k}\omega\right\|_{L^6}\left\|\nabla \omega\right\|_{L^3}+
\left\|\omega\right\|_{L^{\infty}}\left\|\nabla^{k+1}\omega\right\|_{L^2}\right)\\
\leq&C\eta_{1}\left(\left\|\nabla^{k+1}n\right\|^{2}_{L^2}+\left\|\nabla^{k+1}\omega\right\|^{2}_{L^2}\right),
\end{split}
\end{equation*}
\begin{equation*}
\begin{split}
|I_{15}|
\leq &C\left\|\nabla^{k+1}n\right\|_{L^2}\left\|\nabla^{k}[h(n)(\mu\Delta \omega+(\mu+\lambda)\nabla{\rm div}\omega)]\right\|_{L^{2}}\\
\leq& C\left\|\nabla^{k+1}n\right\|_{L^2}\left(\left\|\nabla^{k}n\right\|_{L^6}\left\|\nabla^{2} \omega\right\|_{L^3}+
\left\|n\right\|_{L^{\infty}}\left\|\nabla^{k+2}\omega\right\|_{L^2}\right)\\
\leq&C\eta_{1}\left(\left\|\nabla^{k+1}n\right\|^{2}_{L^2}+\left\|\nabla^{k+2}\omega\right\|^{2}_{L^2}\right),
\end{split}
\end{equation*}
\begin{equation*}
\begin{split}
|I_{16}|
\leq &C\left\|\nabla^{k+1}n\right\|_{L^2}\left\|\nabla^{k}(h(n)\nabla n)\right\|_{L^{2}}\\
\leq& C\left\|\nabla^{k+1}n\right\|_{L^2}\left(\left\|\nabla^{k}n\right\|_{L^6}\left\|\nabla n\right\|_{L^3}+
\left\|n\right\|_{L^{\infty}}\left\|\nabla^{k+1}n\right\|_{L^2}\right)\\
\leq&C\eta_{1}\left\|\nabla^{k+1}n\right\|^{2}_{L^2}.
\end{split}
\end{equation*}

Furthermore, applying H\"{o}lder's inequality, Kato-Ponce's inequality, Sobolev embedding theorem and Lemma \ref{lem2-3}, we estimate $I_{17}$ as
\begin{equation*}
\begin{split}
|I_{17}|
\leq &C\left\|\nabla^{k+1}n\right\|_{L^2}\left\|\nabla^{k}(g(n)\vartheta\nabla n)\right\|_{L^{2}}\\
\leq& C\left\|\nabla^{k+1}n\right\|_{L^2}\left(\left\|\nabla^{k}g(n)\right\|_{L^6}\left\|\vartheta\right\|_{L^6}\left\|\nabla n\right\|_{L^6}
+\left\|g(n)\right\|_{L^{\infty}}\left\|\nabla^{k}(\vartheta\nabla n)\right\|_{L^2}\right)\\
\leq& C\left\|\nabla^{k+1}n\right\|_{L^2}\Big(\left\|\nabla^{k+1}n\right\|_{L^2}\left\|\vartheta\right\|_{L^6}\left\|\nabla n\right\|_{L^6}
+\left\|\nabla^{k}\vartheta\right\|_{L^6}\left\|\nabla n\right\|_{L^{3}}
\\&+\left\|\vartheta\right\|_{L^{\infty}}\left\|\nabla^{k+1}n\right\|_{L^2}\Big)\\
\leq&C\eta_{1}\left(\left\|\nabla^{k+1}n\right\|^{2}_{L^2}+\left\|\nabla^{k+1}\vartheta\right\|^{2}_{L^2}\right).
\end{split}
\end{equation*}
Combining $I_{i}$ $(11\leq i\leq17)$, we give \eqref{equ3-9} and complete the proof of Lemma \ref{lem3-3}.
\end{proof}

\noindent \textbf{Proof of Theorem \ref{the1-1}}. Suppose that $N\geq3$, summing up the estimate \eqref{equ3-6} of Lemma \ref{lem3-1} from $k=l$ to $m-1$ $(1\leq m\leq N)$, then adding \eqref{equ3-8} of Lemma \ref{lem3-2} for $k=m$ to the results, we have
\begin{equation}\label{equ3-11}
\begin{split}
&\frac{d}{dt}\sum_{l\leq k\leq m}\left\|\nabla^{k}(n,\omega,\vartheta)\right\|^2_{L^2}
+\sum_{l\leq k\leq m}\left\|\nabla^{k+1}(\omega,\vartheta)\right\|^2_{L^2}\\
\leq&C\eta_{1}\sum_{l+1\leq k\leq m}\left\|\nabla^{k} n\right\|^2_{L^2}+
C\eta_{1}\sum_{l\leq k\leq m}\left\|\nabla^{k+1} (\omega,\vartheta)\right\|^2_{L^2}.
\end{split}
\end{equation}
Multiplying \eqref{equ3-9} of Lemma \ref{lem3-3} by $\varepsilon$, summing up the resulting equation from $k=l$ to $m-1$ and
combining \eqref{equ3-11}, then choosing $\varepsilon,\eta_{1}>0$ to be small, there holds
\begin{equation}\label{equ3-12}
\frac{d}{dt}\Psi^{m}_{l}(t)+\Gamma^{m}_{l}(t)\leq0,
\end{equation}
where
\begin{equation*}
\Psi^{m}_{l}(t)\simeq \left\|\nabla^{l}n\right\|^2_{H^{m-l}}+\left\|\nabla^{l}\omega\right\|^2_{H^{m-l}}+\left\|\nabla^{l}\vartheta\right\|^2_{H^{m-l}},
\end{equation*}
and
\begin{equation*}
\Gamma^{m}_{l}(t)\simeq \left\|\nabla^{l+1} n\right\|^2_{H^{m-l-1}}+\left\|\nabla^{l+1} \omega\right\|^2_{H^{m-l}}+\left\|\nabla^{l+1}\vartheta\right\|^2_{H^{m-l}}.
\end{equation*}
Let $l=0$, $m=3$, we have
\begin{align*}
\quad \left\|n\right\|^{2}_{H^3}+\left\|\omega\right\|^{2}_{H^{3}}+\left\|\vartheta\right\|^{2}_{H^{3}}
\leq C(\left\|n_{0}\right\|^{2}_{H^3}+\left\|\omega_{0}\right\|^{2}_{H^{3}}+\left\|\vartheta_{0}\right\|^{2}_{H^{3}})\leq C\eta_{0},
\end{align*}
due to the smallness of $\eta_{0}>0$, which yields the priori estimate \eqref{equ3-5}. Let $l=0$ and $m=N$ in \eqref{equ3-12}, and
then integrate it directly in time to obtain \eqref{equ1-5}.

\section{The proof of Theorem \ref{the1-2}}\label{sec4}
In this section, we derive the evolution of the negative Sobolev norms of the solution and obtain the decay estimate \eqref{equ1-7} for $l=0,\cdot\cdot\cdot N-1$. Subsequently, we combine our derived results to establish the decay estimate for $l=N$.

\begin{lemma}\label{lem4-1}
For $s\in[0,\frac{1}{2}]$, we have
\begin{equation}\label{equ4-1}
\begin{split}
&\frac{d}{dt}\left(\left\|\Lambda^{-s}n\right\|^2_{L^2}+\left\|\Lambda^{-s}\omega\right\|^2_{L^2}+
\left\|\Lambda^{-s}\vartheta\right\|^2_{L^2}\right)
\leq C\Big(\left\|\nabla n\right\|^2_{H^1}+\left\|\nabla \omega\right\|^2_{H^1}\\&+\left\|\nabla \vartheta\right\|^2_{H^1}\Big)
\Big(\left\|\Lambda^{-s}n\right\|_{L^2}+\left\|\Lambda^{-s}\omega\right\|_{L^2}
+\left\|\Lambda^{-s}\vartheta\right\|_{L^2}\Big).
\end{split}
\end{equation}
Moreover, for $s\in(\frac{1}{2},\frac{3}{2})$, we have
\begin{equation}\label{equ4-2}
\begin{split}
&\frac{d}{dt}\left(\left\|\Lambda^{-s}n\right\|^2_{L^2}+\left\|\Lambda^{-s}\omega\right\|^2_{L^2}+
\left\|\Lambda^{-s}\vartheta\right\|^2_{L^2}\right)\\
\leq& C\left(\left\|n\right\|_{L^2}+\left\| \omega\right\|_{L^2}+\left\| \vartheta\right\|_{L^2}\right)^{s-\frac{1}{2}}
\left(\left\|\nabla n\right\|_{L^2}+\left\|\nabla \omega\right\|_{L^2}+\left\| \nabla\vartheta\right\|_{L^2}\right)^{\frac{3}{2}-s}\cdot\\
&\left(\left\|\nabla n\right\|_{L^2}+\left\|\nabla \omega\right\|_{L^2}+\left\| \nabla\vartheta\right\|_{L^2}+\left\|\nabla^{2} \omega\right\|_{L^2}+\left\|\Delta\vartheta\right\|_{L^2}\right)
\cdot\\
&\left(\left\|\Lambda^{-s}n\right\|_{L^2}+\left\|\Lambda^{-s}\omega\right\|_{L^2}
+\left\|\Lambda^{-s}\vartheta\right\|_{L^2}\right).
\end{split}
\end{equation}
\end{lemma}
\begin{proof}
Applying $\Lambda^{-s}$ to $\eqref{equ3-3}_{1}$, $\eqref{equ3-3}_{2}$, $\eqref{equ3-3}_{3}$, multiplying the resulting equation by $\Lambda^{-s}n$, $\Lambda^{-s}\omega$, $\Lambda^{-s}\vartheta$, and integrating over $\mathbb{R}^3$, we have
\begin{equation}\label{equ4-3}
\begin{split}
&\frac{1}{2}\frac{d}{dt}\left(\left\|\Lambda^{-s}n\right\|^2_{L^2}+\left\|\Lambda^{-s}\omega\right\|^2_{L^2}
+\left\|\Lambda^{-s}\vartheta\right\|^2_{L^2}\right)\\
&+\mu\left\|\Lambda^{-s}\nabla \omega\right\|^2_{L^2}+(\mu+\lambda)\left\|\Lambda^{-s}{\rm div}\omega\right\|^2_{L^2}+\kappa\left\|\Lambda^{-s}\nabla\vartheta\right\|^2_{L^2}\\
=&-c\int_{\mathbb{R}^3}\Lambda^{-s}(n{\rm div}\omega)\cdot\Lambda^{-s}ndx
-c\int_{\mathbb{R}^3}\Lambda^{-s}(\omega\cdot\nabla n)\cdot\Lambda^{-s}ndx\\
&-c\int_{\mathbb{R}^3}\Lambda^{-s}(\omega\cdot\nabla \omega)\cdot\Lambda^{-s}\omega dx
-\int_{\mathbb{R}^3}\Lambda^{-s}\left[h(n)(\mu\Delta \omega+(\mu+\lambda)\nabla{\rm div}\omega)\right]\cdot\Lambda^{-s}\omega dx\\
&+c\int_{\mathbb{R}^3}\Lambda^{-s}(h(n)\nabla n)\cdot\Lambda^{-s}\omega dx
-\sigma\int_{\mathbb{R}^3}\Lambda^{-s}(g(n)\vartheta\nabla n)\cdot\Lambda^{-s}\omega dx\\
&-c\int_{\mathbb{R}^3}\Lambda^{-s}(\omega\cdot\nabla\vartheta)\cdot\Lambda^{-s}\vartheta dx
-\frac{\sigma^{2}}{c}\int_{\mathbb{R}^3}\Lambda^{-s}(\vartheta{\rm div}\omega)\cdot\Lambda^{-s}\vartheta dx\\
&+\frac{\sigma}{c}\int_{\mathbb{R}^3}\Lambda^{-s}\left[g(n)(2\mu|D \omega|^{2}+\lambda({\rm div}\omega)^{2})\right]\cdot\Lambda^{-s}\vartheta dx
-\kappa\int_{\mathbb{R}^3}\Lambda^{-s}(h(n)\Delta\vartheta)\cdot\Lambda^{-s}\vartheta dx\\
:=&\sum_{i=1}^{10}J_{i}.
\end{split}
\end{equation}

In order to estimate $J_{i}$ $(1\leq i\leq10)$, we require $s\in(0,\frac{3}{2})$. If $s\in(0,\frac{1}{2}]$, we obtain $\frac{1}{2}+\frac{s}{3}<1$, $\frac{3}{s}\geq6$, by using H\"{o}lder's inequality, Gagliardo-Nirenberg's inequality together with Lemma \ref{lem2-4}, it yields that
\begin{equation*}
\begin{split}
|J_{1}|
\leq& C\left\|\Lambda^{-s}n\right\|_{L^2}\left\|n{\rm div}\omega\right\|_{L^{\frac{1}{\frac{1}{2}+\frac{s}{3}}}}\\
\leq& C\left\|\Lambda^{-s}n\right\|_{L^2}\left\|n\right\|_{L^{\frac{3}{s}}}\left\|\nabla \omega\right\|_{L^2}\\
\leq& C\left\|\Lambda^{-s}n\right\|_{L^2}\left\|\nabla n\right\|^{\frac{1}{2}+s}_{L^{2}}
\left\|\nabla^{2} n\right\|^{\frac{1}{2}-s}_{L^{2}}\left\|\nabla \omega\right\|_{L^2},
\end{split}
\end{equation*}
\begin{equation*}
\begin{split}
|J_{2}|
\leq& C\left\|\Lambda^{-s}n\right\|_{L^2}\left\|\omega\cdot\nabla n\right\|_{L^{\frac{1}{\frac{1}{2}+\frac{s}{3}}}}\\
\leq& C\left\|\Lambda^{-s}n\right\|_{L^2}\left\|\omega\right\|_{L^{\frac{3}{s}}}\left\|\nabla n\right\|_{L^2}\\
\leq& C\left\|\Lambda^{-s}n\right\|_{L^2}\left\|\nabla \omega\right\|^{\frac{1}{2}+s}_{L^{2}}\left\|\nabla^{2} \omega\right\|^{\frac{1}{2}-s}_{L^{2}}\left\|\nabla n\right\|_{L^2},
\end{split}
\end{equation*}
\begin{equation*}
\begin{split}
|J_{3}|
\leq& C\left\|\Lambda^{-s}\omega\right\|_{L^2}\left\|\omega\cdot\nabla \omega\right\|_{L^{\frac{1}{\frac{1}{2}+\frac{s}{3}}}}\\
\leq& C\left\|\Lambda^{-s}\omega\right\|_{L^2}\left\|\omega\right\|_{L^{\frac{3}{s}}}\left\|\nabla \omega\right\|_{L^2}\\
\leq& C\left\|\Lambda^{-s}\omega\right\|_{L^2}\left\|\nabla \omega\right\|^{\frac{1}{2}+s}_{L^{2}}
\left\|\nabla^{2} \omega\right\|^{\frac{1}{2}-s}_{L^{2}}\left\|\nabla \omega\right\|_{L^2},
\end{split}
\end{equation*}
\begin{equation*}
\begin{split}
|J_{4}|
\leq& C\left\|\Lambda^{-s}\omega\right\|_{L^2}
\left\|h(n)(\Delta \omega+\nabla{\rm div}\omega)\right\|_{L^{\frac{1}{\frac{1}{2}+\frac{s}{3}}}}\\
\leq& C\left\|\Lambda^{-s}\omega\right\|_{L^2}\left\|h(n)\right\|_{L^{\frac{3}{s}}}\left\|\nabla^{2} \omega\right\|_{L^2}\\
\leq& C\left\|\Lambda^{-s}\omega\right\|_{L^2}\left\|\nabla n\right\|^{\frac{1}{2}+s}_{L^{2}}
\left\|\nabla^{2} n\right\|^{\frac{1}{2}-s}_{L^{2}}\left\|\nabla^{2} \omega\right\|_{L^2},
\end{split}
\end{equation*}
\begin{equation*}
\begin{split}
|J_{5}|
\leq& C\left\|\Lambda^{-s}\omega\right\|_{L^2}\left\|h(n)\nabla n\right\|_{L^{\frac{1}{\frac{1}{2}+\frac{s}{3}}}}\\
\leq& C\left\|\Lambda^{-s}\omega\right\|_{L^2}\left\|h(n)\right\|_{L^{\frac{3}{s}}}\left\|\nabla n\right\|_{L^2}\\
\leq& C\left\|\Lambda^{-s}\omega\right\|_{L^2}\left\|\nabla n\right\|^{\frac{1}{2}+s}_{L^{2}}
\left\|\nabla^{2} n\right\|^{\frac{1}{2}-s}_{L^{2}}\left\|\nabla n\right\|_{L^2},
\end{split}
\end{equation*}
\begin{equation*}
\begin{split}
|J_{6}|
\leq& C\left\|\Lambda^{-s}\omega\right\|_{L^2}\left\|g(n)\vartheta\nabla n\right\|_{L^{\frac{1}{\frac{1}{2}+\frac{s}{3}}}}\\
\leq& C\left\|\Lambda^{-s}\omega\right\|_{L^2}\left\|g(n)\right\|_{L^\infty}\left\|\vartheta\right\|_{L^{\frac{3}{s}}}\left\|\nabla n\right\|_{L^2}\\
\leq& C\left\|\Lambda^{-s}\omega\right\|_{L^2}\left\|\nabla \vartheta\right\|^{\frac{1}{2}+s}_{L^{2}}
\left\|\nabla^{2} \vartheta\right\|^{\frac{1}{2}-s}_{L^{2}}\left\|\nabla n\right\|_{L^2},
\end{split}
\end{equation*}
\begin{equation*}
\begin{split}
|J_{7}|
\leq& C\left\|\Lambda^{-s}\vartheta\right\|_{L^2}\left\|\omega\cdot\nabla\vartheta\right\|_{L^{\frac{1}{\frac{1}{2}+\frac{s}{3}}}}\\
\leq& C\left\|\Lambda^{-s}\vartheta\right\|_{L^2}\left\|\omega\right\|_{L^{\frac{3}{s}}}\left\|\nabla \vartheta\right\|_{L^2}\\
\leq& C\left\|\Lambda^{-s}\vartheta\right\|_{L^2}\left\|\nabla \omega\right\|^{\frac{1}{2}+s}_{L^{2}}\left\|\nabla^{2} \omega\right\|^{\frac{1}{2}-s}_{L^{2}}\left\|\nabla \vartheta\right\|_{L^2},
\end{split}
\end{equation*}
\begin{equation*}
\begin{split}
|J_{8}|
\leq& C\left\|\Lambda^{-s}\vartheta\right\|_{L^2}\left\|\vartheta{\rm div}\omega\right\|_{L^{\frac{1}{\frac{1}{2}+\frac{s}{3}}}}\\
\leq& C\left\|\Lambda^{-s}\vartheta\right\|_{L^2}\left\|\vartheta\right\|_{L^{\frac{3}{s}}}\left\|\nabla \omega\right\|_{L^2}\\
\leq& C\left\|\Lambda^{-s}\vartheta\right\|_{L^2}\left\|\nabla \vartheta\right\|^{\frac{1}{2}+s}_{L^{2}}\left\|\nabla^{2} \vartheta\right\|^{\frac{1}{2}-s}_{L^{2}}\left\|\nabla \omega\right\|_{L^2},
\end{split}
\end{equation*}
\begin{equation*}
\begin{split}
|J_{9}|
\leq& C\left\|\Lambda^{-s}\vartheta\right\|_{L^2}
\left\|g(n)(|D \omega|^{2}+({\rm div}\omega)^{2})\right\|_{L^{\frac{1}{\frac{1}{2}+\frac{s}{3}}}}\\
\leq& C\left\|\Lambda^{-s}\vartheta\right\|_{L^2}\left\|g(n)\right\|_{L^{\frac{3}{s}}}
\left\|\nabla \omega\right\|_{L^6}\left\|\nabla \omega\right\|_{L^3}\\
\leq& C\left\|\Lambda^{-s}\vartheta\right\|_{L^2}\left\|\nabla n\right\|^{\frac{1}{2}+s}_{L^{2}}
\left\|\nabla^{2} n\right\|^{\frac{1}{2}-s}_{L^{2}}\left\|\nabla^{2} \omega\right\|_{L^2},
\end{split}
\end{equation*}
\begin{equation*}
\begin{split}
|J_{10}|
\leq& C\left\|\Lambda^{-s}\vartheta\right\|_{L^2}\left\|h(n)\Delta\vartheta\right\|_{L^{\frac{1}{\frac{1}{2}+\frac{s}{3}}}}\\
\leq& C\left\|\Lambda^{-s}\vartheta\right\|_{L^2}\left\|h(n)\right\|_{L^{\frac{3}{s}}}\left\|\Delta\vartheta \right\|_{L^2}\\
\leq& C\left\|\Lambda^{-s}\vartheta\right\|_{L^2}\left\|\nabla n\right\|^{\frac{1}{2}+s}_{L^{2}}
\left\|\nabla^{2} n\right\|^{\frac{1}{2}-s}_{L^{2}}\left\|\Delta\vartheta\right\|_{L^2}.
\end{split}
\end{equation*}
Combining $J_{i}$ $(1\leq i\leq10)$ and using Young's inequality, the estimate \eqref{equ4-1} follows.

Next, if $s\in(\frac{1}{2},\frac{3}{2})$, it is easy to see that $\frac{1}{2}+\frac{s}{3}<1$, $2<\frac{3}{s}<6$. Then
\begin{equation*}
\begin{split}
|J_{1}|
\leq& C\left\|\Lambda^{-s}n\right\|_{L^2}\left\|n{\rm div}\omega\right\|_{L^{\frac{1}{\frac{1}{2}+\frac{s}{3}}}}\\
\leq& C\left\|\Lambda^{-s}n\right\|_{L^2}\left\|n\right\|_{L^{\frac{3}{s}}}\left\|\nabla \omega\right\|_{L^2}\\
\leq& C\left\|\Lambda^{-s}n\right\|_{L^2}\left\| n\right\|^{s-\frac{1}{2}}_{L^{2}}\left\|\nabla n\right\|^{\frac{3}{2}-s}_{L^{2}}
\left\|\nabla \omega\right\|_{L^2},\\
\end{split}
\end{equation*}
\begin{equation*}
\begin{split}
|J_{2}|
\leq& C\left\|\Lambda^{-s}n\right\|_{L^2}\left\|\omega\cdot\nabla n\right\|_{L^{\frac{1}{\frac{1}{2}+\frac{s}{3}}}}\\
\leq& C\left\|\Lambda^{-s}n\right\|_{L^2}\left\|\omega\right\|_{L^{\frac{3}{s}}}\left\|\nabla n\right\|_{L^2}\\
\leq& C\left\|\Lambda^{-s}n\right\|_{L^2}\left\| \omega\right\|^{s-\frac{1}{2}}_{L^{2}}\left\|\nabla \omega\right\|^{\frac{3}{2}-s}_{L^{2}}\left\|\nabla n\right\|_{L^2},\\
\end{split}
\end{equation*}
\begin{equation*}
\begin{split}
|J_{3}|
\leq& C\left\|\Lambda^{-s}\omega\right\|_{L^2}\left\|\omega\cdot\nabla \omega\right\|_{L^{\frac{1}{\frac{1}{2}+\frac{s}{3}}}}\\
\leq& C\left\|\Lambda^{-s}\omega\right\|_{L^2}\left\|\omega\right\|_{L^{\frac{3}{s}}}\left\|\nabla \omega\right\|_{L^2}\\
\leq& C\left\|\Lambda^{-s}\omega\right\|_{L^2}\left\| \omega\right\|^{s-\frac{1}{2}}_{L^{2}}\left\|\nabla \omega\right\|^{\frac{3}{2}-s}_{L^{2}}
\left\|\nabla \omega\right\|_{L^2},
\end{split}
\end{equation*}
\begin{equation*}
\begin{split}
|J_{4}|
\leq& C\left\|\Lambda^{-s}\omega\right\|_{L^2}\left\|h(n)(\Delta \omega+\nabla{\rm div}\omega)\right\|_{L^{\frac{1}{\frac{1}{2}+\frac{s}{3}}}}\\
\leq& C\left\|\Lambda^{-s}\omega\right\|_{L^2}\left\|h(n)\right\|_{L^{\frac{3}{s}}}\left\|\nabla^{2} \omega\right\|_{L^2}\\
\leq& C\left\|\Lambda^{-s}\omega\right\|_{L^2}\left\| n\right\|^{s-\frac{1}{2}}_{L^{2}}\left\|\nabla n\right\|^{\frac{3}{2}-s}_{L^{2}}
\left\|\nabla^{2} \omega\right\|_{L^2},
\end{split}
\end{equation*}
\begin{equation*}
\begin{split}
|J_{5}|
\leq& C\left\|\Lambda^{-s}\omega\right\|_{L^2}\left\|h(n)\nabla n\right\|_{L^{\frac{1}{\frac{1}{2}+\frac{s}{3}}}}\\
\leq& C\left\|\Lambda^{-s}\omega\right\|_{L^2}\left\|h(n)\right\|_{L^{\frac{3}{s}}}\left\|\nabla n\right\|_{L^2}\\
\leq& C\left\|\Lambda^{-s}\omega\right\|_{L^2}\left\| n\right\|^{s-\frac{1}{2}}_{L^{2}}
\left\|\nabla n\right\|^{\frac{3}{2}-s}_{L^{2}}\left\|\nabla n\right\|_{L^2},
\end{split}
\end{equation*}
\begin{equation*}
\begin{split}
|J_{6}|
\leq& C\left\|\Lambda^{-s}\omega\right\|_{L^2}\left\|g(n)\vartheta\nabla n\right\|_{L^{\frac{1}{\frac{1}{2}+\frac{s}{3}}}}\\
\leq& C\left\|\Lambda^{-s}\omega\right\|_{L^2}\left\|g(n)\right\|_{L^{\infty}}
\left\|\vartheta\right\|_{L^{\frac{3}{s}}}\left\|\nabla n\right\|_{L^2}\\
\leq& C\left\|\Lambda^{-s}\omega\right\|_{L^2}\left\| \vartheta\right\|^{s-\frac{1}{2}}_{L^{2}}\left\|\nabla\vartheta\right\|^{\frac{3}{2}-s}_{L^{2}}
\left\|\nabla n\right\|_{L^2},
\end{split}
\end{equation*}
\begin{equation*}
\begin{split}
|J_{7}|
\leq& C\left\|\Lambda^{-s}\vartheta\right\|_{L^2}\left\|\omega\cdot\nabla\vartheta\right\|_{L^{\frac{1}{\frac{1}{2}+\frac{s}{3}}}}\\
\leq& C\left\|\Lambda^{-s}\vartheta\right\|_{L^2}\left\|\omega\right\|_{L^{\frac{3}{s}}}\left\|\nabla \vartheta\right\|_{L^2}\\
\leq& C\left\|\Lambda^{-s}\vartheta\right\|_{L^2}\left\|\omega\right\|^{s-\frac{1}{2}}_{L^{2}}\left\|\nabla \omega\right\|^{\frac{3}{2}-s}_{L^{2}}\left\|\nabla \vartheta\right\|_{L^2},
\end{split}
\end{equation*}
\begin{equation*}
\begin{split}
|J_{8}|
\leq& C\left\|\Lambda^{-s}\vartheta\right\|_{L^2}\left\|\vartheta{\rm div}\omega\right\|_{L^{\frac{1}{\frac{1}{2}+\frac{s}{3}}}}\\
\leq& C\left\|\Lambda^{-s}\vartheta\right\|_{L^2}\left\|\vartheta\right\|_{L^{\frac{3}{s}}}\left\|\nabla \omega\right\|_{L^2}\\
\leq& C\left\|\Lambda^{-s}\vartheta\right\|_{L^2}\left\|\vartheta\right\|^{s-\frac{1}{2}}_{L^{2}}\left\|\nabla \vartheta\right\|^{\frac{3}{2}-s}_{L^{2}}\left\|\nabla \omega\right\|_{L^2},
\end{split}
\end{equation*}
\begin{equation*}
\begin{split}
|J_{9}|
\leq& C\left\|\Lambda^{-s}\vartheta\right\|_{L^2}\left\|g(n)( |D \omega|^{2}+({\rm div}\omega)^{2})\right\|_{L^{\frac{1}{\frac{1}{2}+\frac{s}{3}}}}\\
\leq& C\left\|\Lambda^{-s}\vartheta\right\|_{L^2}\left\|g(n)\right\|_{L^{\frac{3}{s}}}
\left\|\nabla \omega\right\|_{L^6}\left\|\nabla \omega\right\|_{L^3}\\
\leq& C\left\|\Lambda^{-s}\vartheta\right\|_{L^2}\left\| n\right\|^{s-\frac{1}{2}}_{L^{2}}
\left\|\nabla n\right\|^{\frac{3}{2}-s}_{L^{2}}\left\|\nabla^{2}\omega\right\|_{L^2},
\end{split}
\end{equation*}
\begin{equation*}
\begin{split}
|J_{10}|
\leq& C\left\|\Lambda^{-s}\vartheta\right\|_{L^2}\left\|h(n)\Delta\vartheta\right\|_{L^{\frac{1}{\frac{1}{2}+\frac{s}{3}}}}\\
\leq& C\left\|\Lambda^{-s}\vartheta\right\|_{L^2}\left\|h(n)\right\|_{L^{\frac{3}{s}}}\left\|\Delta\vartheta \right\|_{L^2}\\
\leq& C\left\|\Lambda^{-s}\vartheta\right\|_{L^2}\left\| n\right\|^{s-\frac{1}{2}}_{L^{2}}\left\|\nabla n\right\|^{\frac{3}{2}-s}_{L^{2}}
\left\|\Delta\vartheta\right\|_{L^2}.
\end{split}
\end{equation*}
Combining $J_{i}$ $(1\leq i\leq10)$, the estimate \eqref{equ4-2} follows.
\end{proof}

\noindent \textbf{Proof of Theorem \ref{the1-2}}. \noindent \textbf{Step 1}. We first prove the decay rate of solutions for $s\in[0,\frac{1}{2}]$.  Define
\begin{equation*}
\begin{split}
M_{-s}(t)=\left\|\Lambda^{-s}n\right\|^2_{L^2}+\left\|\Lambda^{-s}\omega\right\|^2_{L^2}+
\left\|\Lambda^{-s}\vartheta\right\|^2_{L^2}.
\end{split}
\end{equation*}
Then, integrating \eqref{equ4-1} in time, we have
\begin{equation}\label{equ4-4}
\begin{split}
M_{-s}(t)\leq &M_{-s}(0)+C\int_0^t\left(\left\|\nabla n\right\|^2_{H^1}+\left\|\nabla \omega\right\|^2_{H^1}+\left\|\nabla \vartheta\right\|^2_{H^1}\right)\sqrt{M_{-s}(\tau)}d\tau\\
\leq& C\Big(1+\sup_{0\leq\tau\leq t}\sqrt{M_{-s}(\tau)}\Big),
\end{split}
\end{equation}
which implies \eqref{equ1-6} for $s\in[0,\frac{1}{2}]$, that is
\begin{equation*}
\left\|\Lambda^{-s}n\right\|^2_{L^2}+\left\|\Lambda^{-s}\omega\right\|^2_{L^2}+
\left\|\Lambda^{-s}\vartheta\right\|^2_{L^2}\leq C.
\end{equation*}
Moreover, thanks to Lemma \ref{lem2-5}, we have
\begin{equation*}
\left\|\nabla^{l+1} f\right\|_{L^2}\geq C\left\|\Lambda^{-s}f\right\|^{-\frac{1}{l+s}}_{L^2}\left\|\nabla^{l} f\right\|^{1+\frac{1}{l+s}}_{L^2}.
\end{equation*}
Then, we get
\begin{equation*}
\left\|\nabla^{l+1} (n,\omega,\vartheta)\right\|_{L^2}\geq C\left\|\nabla^{l} (n,\omega,\vartheta)\right\|^{1+\frac{1}{l+s}}_{L^2}.
\end{equation*}
Also, due to $\left\|\nabla^{k} n\right\|^{2}_{L^{2}}\leq C$, we have
\begin{equation*}
\left(\left\|\nabla^{k} n\right\|^{2}_{L^{2}}\right)^{1+\frac{1}{l+s}}
=\left\|\nabla^{k} n\right\|^{2}_{L^{2}}\left\|\nabla^{k} n\right\|^{\frac{2}{l+s}}_{L^{2}}
\leq C\left\|\nabla^{k} n\right\|^{2}_{L^{2}}.
\end{equation*}
So, we obtain
\begin{equation*}
\begin{split}
&\left\|\nabla^{l+1} n\right\|^{2}_{H^{N-l-1}}+\left\|\nabla^{l+1} \omega\right\|^{2}_{H^{N-l}}
+\left\|\nabla^{l+1} \vartheta\right\|^{2}_{H^{N-l}}\\
\geq &\frac{1}{2}\left(\left\|\nabla^{l+1} n\right\|^{2}_{H^{N-l-1}}+\left\|\nabla^{l+1} \omega\right\|^{2}_{H^{N-l}}+\left\|\nabla^{l+1} \vartheta\right\|^{2}_{H^{N-l}}\right)+\frac{1}{2}\left\|\nabla^{l+1} n\right\|^{2}_{L^{N-l-1}}\\
\geq& C_{0}\left(\left\|\nabla^{l} n\right\|^{2}_{H^{N-l-1}}+\left\|\nabla^{l} \omega\right\|^{2}_{H^{N-l}}+
\left\|\nabla^{l} \vartheta\right\|^{2}_{H^{N-l}}\right)^{1+\frac{1}{l+s}}
+\frac{1}{2}\left\|\nabla^{N} n\right\|^{2}_{L^{2}}\\
\geq& C_{0}\left(\left\|\nabla^{l} n\right\|^{2}_{H^{N-l}}+\left\|\nabla^{l}\omega\right\|^{2}_{H^{N-l}}+
\left\|\nabla^{l} \vartheta\right\|^{2}_{H^{N-l}}\right)^{1+\frac{1}{l+s}},
\end{split}
\end{equation*}
where $C_{0}\leq\frac{1}{2}$ is a fixed positive constant. Thus,  by \eqref{equ3-12}, we have
\begin{equation*}
\frac{d}{dt}\Psi_{l}^{N}+C_{0}\left(\Psi_{l}^{N}\right)^{1+\frac{1}{l+s}}\leq 0,
\end{equation*}
this means that
\begin{equation*}
\Psi_{l}^{N}\leq C(t+1)^{-(l+s)},\quad l=0,1,\cdot\cdot\cdot,N-1.
\end{equation*}

\noindent \textbf{Step 2}. We next prove the decay rate of solutions for $s\in(\frac{1}{2},\frac{3}{2})$. The arguments for the case $s\in[0,\frac{1}{2}]$ can not be applied to this case. Notice that $(n_{0},\omega_{0},\vartheta_{0})\in \dot{H}^{-s}\cap L^{2}\subset\dot{H}^{-\frac{1}{2}}$ when $s\in(\frac{1}{2},\frac{3}{2})$, we can deduce from \eqref{equ1-7} with $s=\frac{1}{2}$, there holds
\begin{equation}\label{equ4-5}
\left\|\nabla^{l}n\right\|^2_{H^{N-l}}+\left\|\nabla^{l}\omega\right\|^2_{H^{N-l}}+
\left\|\nabla^{l}\vartheta\right\|^2_{H^{N-l}}\leq C(t+1)^{-\frac{1}{2}-l}.
\end{equation}
Therefore, we deduce from \eqref{equ4-2} that for $s\in(\frac{1}{2},\frac{3}{2})$,
\begin{equation}\label{equ4-6}
\begin{split}
M_{-s}(t)\leq &M_{-s}(0)+C\int_0^t(\left\|n\right\|_{L^2}+\left\| \omega\right\|_{L^2}+\left\|\vartheta\right\|_{L^2})^{s-\frac{1}{2}}\cdot\\
&(\left\|\nabla n\right\|_{L^2}+\left\|\nabla \omega\right\|_{L^2}+\left\| \nabla\vartheta\right\|_{L^2})^{\frac{3}{2}-s}\cdot(\left\|\nabla n\right\|_{L^2}+\left\|\nabla \omega\right\|_{L^2}\\&+\left\| \nabla\vartheta\right\|_{L^2}+\left\|\nabla^{2} \omega\right\|_{L^2}+\left\|\Delta\vartheta\right\|_{L^2})\sqrt{M_{-s}(\tau)}d\tau\\
\leq& C+C\int_0^t(\tau+1)^{-\frac{7}{4}+\frac{s}{2}}d\tau\cdot\sup_{0\leq\tau\leq t}\sqrt{M_{-s}(\tau)}\\
\leq& C\Big(1+\sup_{0\leq\tau\leq t}\sqrt{M_{-s}(\tau)}\Big),
\end{split}
\end{equation}
which implies
\begin{equation*}
\left\|\Lambda^{-s}n\right\|^2_{L^2}+\left\|\Lambda^{-s}\omega\right\|^2_{L^2}+
\left\|\Lambda^{-s}\vartheta\right\|^2_{L^2}\leq C,\quad s\in(\frac{1}{2},\frac{3}{2}).
\end{equation*}
Then, we may repeat the arguments leading to \eqref{equ1-7} for $s\in[0,\frac{1}{2}]$ to prove that they hold also for $s\in(\frac{1}{2},\frac{3}{2})$.

\noindent \textbf{Step 3}. We will establish the optimal decay of the $N-th$ order spatial derivative.
\begin{lemma}\label{lem4-2}
Under the assumption of Theorem \ref{the1-2}, it holds that
\begin{equation}\label{equ4-7}
\begin{split}
&(t+1)^{N+s-1+\epsilon_{0}}\left\|\nabla^{N-1}(n,\omega,\vartheta)\right\|^2_{H^{1}}
+\\&\int_0^t(\tau+1)^{N+s-1+\epsilon_{0}}\left(\left\|\nabla^{N}n\right\|^2_{L^{2}}
+\left\|\nabla^{N}\omega\right\|^2_{H^{1}}+\left\|\nabla^{N}\vartheta\right\|^2_{H^{1}}\right)d\tau\\
\leq& C(t+1)^{\epsilon_{0}}.
\end{split}
\end{equation}
\end{lemma}
\begin{proof}
Multiplying the inequality \eqref{equ3-12} by $(t+1)^{N+s-1+\epsilon_{0}}$, taking $l=N-1$, $m=N$, and combining the decay estimate \eqref{equ1-7},
it holds true
\begin{equation*}
\begin{split}
&\frac{d}{dt}\left[(t+1)^{N+s-1+\epsilon_{0}}\Psi_{N-1}^{N}(t)\right]+(t+1)^{N+s-1+\epsilon_{0}}\Gamma_{N-1}^{N}(t)\\
\leq& C(t+1)^{N+s-2+\epsilon_{0}}\Psi_{N-1}^{N}(t)\\
\leq& C(t+1)^{N+s-2+\epsilon_{0}}\left\|\nabla^{N-1}(n,\omega,\vartheta)\right\|^2_{H^{1}}\\
\leq& C(t+1)^{-1+\epsilon_{0}}.
\end{split}
\end{equation*}
Integrating the above inequality in time, the estimate \eqref{equ4-7} yields.
\end{proof}

Next, by H\"{o}lder's inequality and Kato-Ponce's inequality, we can deduce
\begin{equation*}
\begin{split}
|I_{1}|\leq& C\left\|\nabla^{N}n\right\|_{L^2}\left\|\nabla^{N}(n{\rm div}\omega)\right\|_{L^2}\\
\leq& C\left\|\nabla^{N}n\right\|_{L^2}\left(\left\|\nabla^{N}n\right\|_{L^2}\left\|{\rm div} \omega\right\|_{L^\infty}+
\left\|n\right\|_{L^\infty}\left\|\nabla^{N}{\rm div}\omega\right\|_{L^2}\right)\\
\leq& C\left\|\nabla^{2}\omega\right\|^{2}_{H^1}\left\|\nabla^{N}n\right\|^{2}_{L^2}
+C\left\|\nabla n\right\|^{2}_{H^1}\left\|\nabla^{N}n\right\|^{2}_{L^2}
+\epsilon\left\|\nabla^{N+1}\omega\right\|_{L^2}\\
\leq&C(t+1)^{-(1+s)}\left\|\nabla^{N}n\right\|^{2}_{L^2}+\epsilon\left\|\nabla^{N+1}\omega\right\|_{L^2}.
\end{split}
\end{equation*}

For the term $I_{2}$ and $I_{3}$, integrating by parts, we find that
\begin{equation*}
\begin{split}
|I_{2}|\leq& C\left|\int_{\mathbb{R}^3}{\rm div}\omega|\nabla^{N}n|^{2}dx\right|+
C\left|\int_{\mathbb{R}^3}\nabla^{N-1}(\nabla\omega\cdot\nabla n)\cdot\nabla^{N}ndx\right|\\
\leq& C\left\|\nabla^{N}n\right\|_{L^2}\left(\left\|\nabla\omega\right\|_{L^\infty}\left\|\nabla^{N}n\right\|_{L^2}+
\left\|\nabla n\right\|_{L^\infty}\left\|\nabla^{N}\omega\right\|_{L^2}\right)\\
\leq& C\left\|\nabla^{N}n\right\|_{L^2}\left(\left\|\nabla^{2}\omega\right\|_{H^1}\left\|\nabla^{N}n\right\|_{L^2}
+\left\|\nabla^{2} n\right\|_{H^1}\left\|\nabla^{N}\omega\right\|_{L^2}\right)\\
\leq&C(t+1)^{-(1+\frac{s}{2})}\left(\left\|\nabla^{N}n\right\|^{2}_{L^2}+\left\|\nabla^{N}\omega\right\|^{2}_{L^2}\right),
\end{split}
\end{equation*}
and
\begin{equation*}
\begin{split}
|I_{3}|\leq& C\left|\int_{\mathbb{R}^3}\nabla^{N-1}(\omega\cdot\nabla\omega)\cdot{\rm div}\nabla^{N}\omega dx\right|\\
\leq& C\left\|\nabla^{N+1}\omega\right\|_{L^2}\left(\left\|\omega\right\|_{L^\infty}\left\|\nabla^{N}\omega\right\|_{L^2}+
\left\|\nabla \omega\right\|_{L^\infty}\left\|\nabla^{N-1}\omega\right\|_{L^2}\right)\\
\leq& C\left\|\nabla^{N+1}\omega\right\|_{L^2}\left(\left\|\nabla\omega\right\|_{H^1}\left\|\nabla^{N}\omega\right\|_{L^2}
+\left\|\nabla^{2} \omega\right\|_{H^1}\left\|\nabla^{N-1}\omega\right\|_{L^2}\right)\\
\leq&C(t+1)^{-(1+s)}\left\|\nabla^{N}\omega\right\|^{2}_{L^2}
+C(t+1)^{-(N+1+2s)}+\epsilon\left\|\nabla^{N+1}\omega\right\|_{L^2}.
\end{split}
\end{equation*}

Given the H\"{o}lder's inequality, Gagliardo-Nirenberg's inequality, we calculate
\begin{equation*}
\begin{split}
|I_{4}|=&\left|\int_{\mathbb{R}^3}\nabla^{N-1}\left[h(n)(\mu\Delta \omega+(\mu+\lambda)\nabla{\rm div}\omega)\right]
\cdot{\rm div}\nabla^{N}\omega dx\right|\\
\leq &C\left\|\nabla^{N+1}\omega\right\|_{L^2}\left\|\nabla^{N-1}(h(n)\nabla^{2}\omega)\right\|_{L^{2}}\\
\leq& C\left\|\nabla^{N+1}\omega\right\|_{L^2}\left(\left\|\nabla^{N-1}n\right\|_{L^2}\left\|\nabla^{2} \omega\right\|_{L^\infty}+
\left\|n\right\|_{L^{\infty}}\left\|\nabla^{N+1}\omega\right\|_{L^2}\right)\\
\leq& C\left\|\nabla^{N+1}\omega\right\|_{L^2}\left(\left\|\nabla^{N-1}n\right\|_{L^2}
\left\|\nabla^{3} \omega\right\|^{\frac{1}{2}}_{L^2}\left\|\nabla^{4} \omega\right\|^{\frac{1}{2}}_{L^2}+
\left\|\nabla n\right\|_{H^{1}}\left\|\nabla^{N+1}\omega\right\|_{L^2}\right)\\
\leq&C(t+1)^{-\frac{N-1+s}{2}}\left\|\nabla^{3} \omega\right\|^{\frac{1}{2}}_{L^2}
\left\|\nabla^{4} \omega\right\|^{\frac{1}{2}}_{L^2}\left\|\nabla^{N+1}\omega\right\|_{L^2}
+C\eta_{1}\left\|\nabla^{N+1}\omega\right\|^{2}_{L^2}.
\end{split}
\end{equation*}
In the case of $N=3$, it follows that
\begin{equation*}
\begin{split}
|I_{4}|\leq& C(t+1)^{-\frac{2+s}{2}}\left\|\nabla^{3} \omega\right\|^{\frac{1}{2}}_{L^2}
\left\|\nabla^{4} \omega\right\|^{\frac{3}{2}}_{L^2}+C\eta_{1}\left\|\nabla^{4}\omega\right\|^{2}_{L^2}\\
\leq&C(t+1)^{-2(2+s)}\left\|\nabla^{3} \omega\right\|^{2}_{L^2}+(C\eta_{1}+\epsilon)\left\|\nabla^{4}\omega\right\|^{2}_{L^2}.
\end{split}
\end{equation*}
In the case of $N\geq4$, by applying the decay \eqref{equ1-7}, it holds that
\begin{equation*}
\begin{split}
|I_{4}|\leq C(t+1)^{-(N+2+2s)}+(C\eta_{1}+\epsilon)\left\|\nabla^{N+1}\omega\right\|^{2}_{L^2}.
\end{split}
\end{equation*}

For the term $I_{5}$ and $I_{6}$, invoking H\"{o}lder's inequality and Lemma \ref{lem2-3}, we derive
\begin{equation*}
\begin{split}
|I_{5}|\leq&C\left|\int_{\mathbb{R}^3}\nabla^{N-1}(h(n)\nabla n)\cdot{\rm div}\nabla^{N}\omega dx\right|\\
\leq &C\left\|\nabla^{N+1}\omega\right\|_{L^2}\left\|\nabla^{N-1}(h(n)\nabla n)\right\|_{L^{2}}\\
\leq& C\left\|\nabla^{N+1}\omega\right\|_{L^2}\left(\left\|\nabla^{N-1}n\right\|_{L^2}\left\|\nabla n\right\|_{L^\infty}+
\left\|n\right\|_{L^\infty}\left\|\nabla^{N}n\right\|_{L^2}\right)\\
\leq& C\left(\left\|\nabla^{N-1}n\right\|^{2}_{L^2}\left\|\nabla^{2} n\right\|^{2}_{H^1}+
\left\|\nabla n\right\|^{2}_{H^1}\left\|\nabla^{N}n\right\|^{2}_{L^2}\right)
+\epsilon\left\|\nabla^{N+1}\omega\right\|^{2}_{L^2}\\
\leq&C(t+1)^{-(N+1+2s)}+C(t+1)^{-(1+s)}\left\|\nabla^{N}n\right\|^{2}_{L^2}+\epsilon\left\|\nabla^{N+1}\omega\right\|^{2}_{L^2},
\end{split}
\end{equation*}
and
\begin{equation*}
\begin{split}
|I_{6}|\leq&C\left|\int_{\mathbb{R}^3}\nabla^{N-1}(g(n)\vartheta\nabla n)\cdot{\rm div}\nabla^{N}\omega dx\right|\\
\leq &C\left\|\nabla^{N+1}\omega\right\|_{L^2}\left\|\nabla^{N-1}(g(n)\vartheta\nabla n)\right\|_{L^{2}}\\
\leq& C\left\|\nabla^{N+1}\omega\right\|_{L^2}\Big[\left\|\nabla^{N-1}g(n)\right\|_{L^2}
\left\|\vartheta\right\|_{L^{\infty}}\left\|\nabla n\right\|_{L^{\infty}}+\\&
\left\|g(n)\right\|_{L^\infty}\left(\left\|\vartheta\right\|_{L^\infty}\left\|\nabla^{k}n\right\|_{L^2}+
\left\|\nabla^{N-1}\vartheta\right\|_{L^2}\left\|\nabla n\right\|_{L^\infty}\right)\Big]\\
\leq& C\Big(\left\|\nabla^{N-1}n\right\|^{2}_{L^2}\left\|\nabla \vartheta\right\|^{2}_{H^1}\left\|\nabla^{2} n\right\|^{2}_{H^1}
+\left\|\nabla\vartheta\right\|^{2}_{H^1}\left\|\nabla^{N}n\right\|^{2}_{L^2}\\
&+\left\|\nabla^{N-1}\vartheta\right\|^{2}_{L^2}\left\|\nabla^{2}n\right\|^{2}_{H^1}\Big)
+\epsilon\left\|\nabla^{N+1}\omega\right\|^{2}_{L^2}\\
\leq&C(t+1)^{-(N+1+2s)}+C(t+1)^{-(1+s)}\left\|\nabla^{N}n\right\|^{2}_{L^2}+\epsilon\left\|\nabla^{N+1}\omega\right\|^{2}_{L^2}.
\end{split}
\end{equation*}

Integrating by parts, then employing once again H\"{o}lder's inequality and Kato-Ponce's inequality, we estimate $I_{7}$ and $I_{8}$ as
\begin{equation*}
\begin{split}
|I_{7}|\leq&C\left|\int_{\mathbb{R}^3}\nabla^{N-1}(\omega\cdot\nabla\vartheta)\cdot{\rm div}\nabla^{N}\vartheta dx\right|\\
\leq &C\left\|\nabla^{N+1}\vartheta\right\|_{L^2}\left\|\nabla^{N-1}(\omega\cdot\nabla\vartheta)\right\|_{L^{2}}\\
\leq& C\left\|\nabla^{N+1}\vartheta\right\|_{L^2}\left(\left\|\nabla^{N-1}\omega\right\|_{L^2}
\left\|\nabla\vartheta\right\|_{L^{\infty}}+\left\|\omega\right\|_{L^\infty}\left\|\nabla^{N}\vartheta\right\|_{L^2}\right)\\
\leq&C \left\|\nabla^{N+1}\vartheta\right\|_{L^2}\left(\left\|\nabla^{N-1}\omega\right\|_{L^2}
\left\|\nabla^{2}\vartheta\right\|_{H^{1}}+\left\|\nabla\omega\right\|_{H^1}\left\|\nabla^{N}\vartheta\right\|_{L^2}\right)\\
\leq&C(t+1)^{-(N+1+2s)}+C(t+1)^{-(1+s)}\left\|\nabla^{N}\vartheta\right\|^{2}_{L^2}+\epsilon\left\|\nabla^{N+1}\vartheta\right\|^{2}_{L^2},
\end{split}
\end{equation*}
and
\begin{equation*}
\begin{split}
|I_{8}|\leq&C\left|\int_{\mathbb{R}^3}\nabla^{N-1}(\vartheta{\rm div}\omega)\cdot{\rm div}\nabla^{N}\vartheta dx\right|\\
\leq &C\left\|\nabla^{N+1}\vartheta\right\|_{L^2}\left\|\nabla^{N-1}(\vartheta{\rm div}\omega)\right\|_{L^{2}}\\
\leq& C\left\|\nabla^{N+1}\vartheta\right\|_{L^2}\left(\left\|\nabla^{N-1}\vartheta\right\|_{L^2}
\left\|\nabla\omega\right\|_{L^{\infty}}+\left\|\vartheta\right\|_{L^\infty}\left\|\nabla^{N}\omega\right\|_{L^2}\right)\\
\leq&C \left\|\nabla^{N+1}\vartheta\right\|_{L^2}\left(\left\|\nabla^{N-1}\vartheta\right\|_{L^2}
\left\|\nabla^{2}\omega\right\|_{H^{1}}+\left\|\nabla\vartheta\right\|_{H^1}\left\|\nabla^{N}\omega\right\|_{L^2}\right)\\
\leq&C(t+1)^{-(N+1+2s)}+C(t+1)^{-(1+s)}\left\|\nabla^{N}\omega\right\|^{2}_{L^2}+\epsilon\left\|\nabla^{N+1}\vartheta\right\|^{2}_{L^2}.
\end{split}
\end{equation*}

Similarly, by virtue of H\"{o}lder's inequality, Kato-Ponce's inequality and Lemma \ref{lem2-3}, it holds
\begin{equation*}
\begin{split}
|I_{9}|\leq&C\left|\int_{\mathbb{R}^3}\nabla^{N-1}(g(n)|\nabla\omega|^{2})\cdot{\rm div}\nabla^{N}\vartheta dx\right|\\
\leq &C\left\|\nabla^{N+1}\vartheta\right\|_{L^2}\left\|\nabla^{N-1}(g(n)|\nabla\omega|^{2})\right\|_{L^{2}}\\
\leq& C\left\|\nabla^{N+1}\vartheta\right\|_{L^2}\left(\left\|\nabla^{N-1}n\right\|_{L^2}
\left\|\nabla\omega\right\|^{2}_{L^{\infty}}+\left\|\nabla\omega\right\|_{L^\infty}
\left\|\nabla^{N}\omega\right\|_{L^2}\right)\\
\leq&C \left\|\nabla^{N+1}\vartheta\right\|_{L^2}\left(\left\|\nabla^{N-1}n\right\|_{L^2}
\left\|\nabla^{2}\omega\right\|^{2}_{H^{1}}+
\left\|\nabla^{2}\omega\right\|_{H^{1}}\left\|\nabla^{N}\omega\right\|_{L^2}\right)\\
\leq&C(t+1)^{-(N+3+3s)}+C(t+1)^{-(2+s)}\left\|\nabla^{N}\omega\right\|^{2}_{L^2}+\epsilon\left\|\nabla^{N+1}\vartheta\right\|^{2}_{L^2},
\end{split}
\end{equation*}
and
\begin{equation*}
\begin{split}
|I_{10}|\leq&C\left|\int_{\mathbb{R}^3}\nabla^{N-1}(h(n)\Delta\vartheta)\cdot{\rm div}\nabla^{N}\vartheta dx\right|\\
\leq& C\left\|\nabla^{N+1}\vartheta\right\|_{L^2}\left\|\nabla^{N-1}(h(n)\Delta\vartheta)\right\|_{L^{2}}\\
\leq& C\left\|\nabla^{N+1}\vartheta\right\|_{L^2}\left(\left\|\nabla^{N-1}n\right\|_{L^2}\left\|\Delta\vartheta\right\|_{L^\infty}+
\left\|n\right\|_{L^{\infty}}\left\|\nabla^{N+1}\vartheta\right\|_{L^2}\right)\\
\leq& C\left\|\nabla^{N+1}\vartheta\right\|_{L^2}\left(\left\|\nabla^{N-1}n\right\|_{L^2}\left\|\nabla^{3}\vartheta\right\|_{H^1}+
\left\|\nabla n\right\|_{H^{1}}\left\|\nabla^{N+1}\vartheta\right\|_{L^2}\right)\\
\leq&C(t+1)^{-(N+2+2s)}+(C\eta_{1}+\epsilon)\left\|\nabla^{N+1}\vartheta\right\|^{2}_{L^2}.
\end{split}
\end{equation*}

Summing up the estimates $I_{i}$ $(1\leq i\leq10)$, and choosing $\eta_{1}$, $\epsilon$ suitably small, we obtain that
\begin{equation}\label{equ4-8}
\begin{split}
&\frac{d}{dt}\left\|\nabla^{N}(n,\omega,\vartheta)\right\|^2_{L^2}+\left\|\nabla^{N+1}(\omega,\vartheta)\right\|^2_{L^2}\\
\leq&C(t+1)^{-1}\left\|\nabla^{N}(n,\omega,\vartheta)\right\|^2_{L^2}+C(t+1)^{-(N+1+s)}.
\end{split}
\end{equation}

By multiplying \eqref{equ4-8} by $(t+1)^{N+s+\epsilon_{0}}$, applying \eqref{equ4-7}, and integrating the result equation in time yields
\begin{equation}\label{equ4-9}
\begin{split}
&(t+1)^{N+s+\epsilon_{0}}\left\|\nabla^{N}(n,\omega,\vartheta)\right\|^2_{L^2}
+\int_0^t(\tau+1)^{N+s+\epsilon_{0}}\left\|\nabla^{N+1}(\omega,\vartheta)\right\|^2_{L^2}d\tau\\
\leq&\left\|\nabla^{N}(n_{0},\omega_{0},\vartheta_{0})\right\|^2_{L^2}
+C\int_0^t(\tau+1)^{N+s-1+\epsilon_{0}}\left\|\nabla^{N}(n,\omega,\vartheta)\right\|^2_{L^2}d\tau
+C\int_0^t(\tau+1)^{-1+\epsilon_{0}}d\tau\\
\leq&\left\|\nabla^{N}(n_{0},\omega_{0},\vartheta_{0})\right\|^2_{L^2}+C(t+1)^{\epsilon_{0}},
\end{split}
\end{equation}
it means that
\begin{equation}\label{equ4-10}
\left\|\nabla^{N}(n,\omega,\vartheta)\right\|^2_{L^{2}}\leq C(t+1)^{-(N+s)}.
\end{equation}
Therefore, the proof of Theorem \ref{the1-2} is complete.   \hfill $\Box$

\section{The proof of Theorem \ref{the1-3}}\label{sec5}
In this section, we will establish the upper and lower bound of the decay rate for the global solution when the initial data $(\rho_{0}-\rho_{\ast},u_{0},\theta_{0}-\theta_{\ast})$ belongs to $H^{N}\cap L^{1}$. To achieve this goal, we need to analyze the linearized system \eqref{equ5-1}.
\begin{equation}\label{equ5-1}
\left\{
  \begin{array}{ll}
    \partial_{t}\tilde{n}+c{\rm div}\tilde{\omega}=0,\\
    \partial_{t}\tilde{\omega}-\mu\Delta \tilde{\omega}-(\mu+\lambda)\nabla{\rm div}\tilde{\omega}+c\nabla \tilde{n}+\sigma\nabla\tilde{\vartheta}=0,\\
    \partial_{t}\tilde{\vartheta}-\kappa\Delta\tilde{\vartheta}+\sigma{\rm div}\tilde{\omega}=0.
     \end{array}
\right.
\end{equation}
We can rewrite the linear system \eqref{equ5-1} as
\begin{equation}\label{equ5-2}
\partial_{t}\tilde{U}=L\tilde{U},\quad \tilde{U}|_{t=0}=U_{0},\quad t\geq0,
\end{equation}
where $U_{0}=(n_{0},\omega_{0},\vartheta_{0})^{tr}$ and the linear operator $L$ is given by
\begin{equation*}
\begin{pmatrix}
0 & -c{\rm div} & 0 \\
-c\nabla & \mu\Delta+(\mu+\lambda)\nabla{\rm div} & -\sigma\nabla \\
0 & -\sigma{\rm div} & \kappa\Delta \\
\end{pmatrix}
\end{equation*}
In addition, the system \eqref{equ3-3} can be rewritten as follows
\begin{equation}\label{equ5-3}
\left\{
  \begin{array}{ll}
    \partial_{t}U=LU+F,\\
     U|_{t=0}=U_{0},
     \end{array}
\right.
\end{equation}
where $F=(S_{1},S_{2},S_{3})^{tr}$. Denote $G(x,t)$ be the Green's function of \eqref{equ5-2}, due to Duhamel's principle, it gives
\begin{equation}\label{equ5-4}
U(x,t)=G\ast U_{0}+\int_0^tG(t-\tau)\ast F(\tau)d\tau.
\end{equation}

Then we introduce the low-high frequency decomposition as follows
\begin{equation}\label{equ5-5}
U(x,t)=U^{l}(x,t)+U^{h}(x,t)=(n^{l},\omega^{l},\vartheta^{l})+(n^{h},\omega^{h},\vartheta^{h}),
\end{equation}
where
\begin{equation*}
U^{l}(x,t)=P_{1}U(x,t),\quad\quad U^{h}(x,t)=P_{\infty}(x,t),
\end{equation*}
and the operator $P_{1}$ and $P_{\infty}$ are introduced in Section \ref{sec1}. It is easy to check that for $0\leq k< m$,
\begin{equation}\label{equ5-6}
\left\|\nabla^{m}U^{l}\right\|_{L^2}\leq C\left\|\nabla^{m}U\right\|_{L^2},\quad\quad
\left\|\nabla^{m}U^{h}\right\|_{L^2}\leq C\left\|\nabla^{m}U\right\|_{L^2},
\end{equation}
\begin{equation}\label{equ5-7}
\left\|\nabla^{m}U^{l}\right\|_{L^2}\leq C\left\|\nabla^{k}U^{l}\right\|_{L^2},\quad\quad
\left\|\nabla^{k}U^{h}\right\|_{L^2}\leq C\left\|\nabla^{m}U^{h}\right\|_{L^2}.
\end{equation}

Next, we define the time-weighted energy function
\begin{equation}\label{equ5-8}
M(t)=\sup_{0\leq\tau\leq t}\left\{(\tau+1)^{\frac{3}{4}}\left\|(n,\omega,\vartheta)\right\|_{H^N}\right\},
\end{equation}
and the following lemmas, which have advantages in establishing the upper and lower bounds of the decay rate.

\begin{lemma} [\label{lem5-1}\cite{CLT}]
For given function $f(x,t)$, we have the following decay estimate
\begin{equation}\label{equ5-9}
\left\|\nabla^{k}G\ast f\right\|_{L^2}\leq C(t+1)^{-\frac{3}{4}-\frac{k}{2}}\left\|f\right\|_{L^1}+
Ce^{-Rt}\left\|\nabla^{k}f\right\|_{L^2}.
\end{equation}
In particular, for $G^{l}(x,t)$, it holds
\begin{equation}\label{equ5-10}
\left\|\nabla^{k}G^{l}\ast f\right\|_{L^2}\leq C(t+1)^{-\frac{3}{4}-\frac{k}{2}}\left\|f\right\|_{L^1}.
\end{equation}
\end{lemma}
\begin{proposition} [\label{the5-2}\cite{CLT}]
Under the assumptions of Theorem \ref{the1-3}, then the linearized problem \eqref{equ5-1} has the decay rates for any large enough $t$,
\begin{equation}\label{equ5-11}
\left\|(\tilde{n},\tilde{\omega},\tilde{\vartheta})(t)\right\|^2_{L^{2}}\geq c_{1}(t+1)^{-\frac{3}{4}},
\end{equation}
where $c_{1}$ s a positive constant independent of time.
\end{proposition}

\noindent \textbf{Proof of Theorem \ref{the1-3}}.  \textbf{Step 1}. We will obtain the decay estimate of the solution for the case $k=0,1$.
\begin{lemma}\label{lem5-3}
Under the assumptions of Theorem \ref{the1-3}, for the case $k=0,1$, then
\begin{equation}\label{equ5-12}
\left\|\nabla^{k}(n,\omega,\vartheta)\right\|_{H^{N-k}}\leq C(t+1)^{-\frac{3}{4}-\frac{k}{2}}.
\end{equation}
\end{lemma}
\begin{proof}
Let $l=0$, $m=N$ in \eqref{equ3-12}, we obtain
\begin{equation}\label{equ5-13}
\frac{d}{dt}\Psi_{0}^{N}(t)+\left\|\nabla n\right\|_{H^{N-1}}+\left\|\nabla \omega\right\|_{H^{N-1}}+\left\|\nabla \vartheta\right\|_{H^{N-1}}\leq0,
\end{equation}
then
\begin{equation*}
\frac{d}{dt}\Psi_{0}^{N}(t)+\Psi_{0}^{N}(t)\leq C\left\|(n,\omega,\vartheta)(t)\right\|_{L^{2}}.
\end{equation*}
Based on the fact
\begin{equation*}
\left\|(n^{h},\omega^{h},\vartheta^{h})(t)\right\|_{L^{2}}\leq C\left\|\nabla(n,\omega,\vartheta)(t)\right\|_{L^{2}},
\end{equation*}
we have
\begin{equation}\label{equ5-14}
\frac{d}{dt}\Psi_{0}^{N}(t)+C_{1}\Psi_{0}^{N}(t)\leq C\left\|(n^{l},\omega^{l},\vartheta^{l})(t)\right\|^{2}_{L^{2}},
\end{equation}
where $C_{1}$ is a positive constant.

According to the \eqref{equ5-4}, we have
\begin{equation}\label{equ5-15}
U^{l}(x,t)=G^{l}\ast U_{0}+\int_0^tG^{l}(t-\tau)\ast F(\tau)d\tau.
\end{equation}
Then, applying Lemma \ref{lem5-1} together with H\"{o}lder's inequality and \eqref{equ5-8}, we obtain
\begin{equation}\label{equ5-16}
\begin{split}
&\left\|(n^{l},\omega^{l},\vartheta^{l})(t)\right\|_{L^{2}}\\
\leq&C(t+1)^{-\frac{3}{4}}\left\|(n_{0},\omega_{0},\vartheta_{0})(t)\right\|_{L^{1}}+C\int_0^t(t-\tau+1)^{-\frac{3}{4}}\left\|F(t)\right\|_{L^{1}}d\tau\\
\leq& CA_{0}(t+1)^{-\frac{3}{4}}+C\int_0^t(t-\tau+1)^{-\frac{3}{4}}\left\|(n,\omega,\vartheta)\right\|_{L^{2}}
(\left\|\nabla n\right\|_{L^{2}}+\left\|\nabla(\omega,\vartheta)\right\|_{H^{1}})d\tau\\
\leq&C(t+1)^{-\frac{3}{4}}(A_{0}+\eta_{0}M(t)).
\end{split}
\end{equation}
By substituting \eqref{equ5-16} into \eqref{equ5-14}, we have
\begin{equation*}
\frac{d}{dt}\Psi_{0}^{N}(t)+C_{1}\Psi_{0}^{N}(t)\leq C(t+1)^{-\frac{3}{2}}\left(A_{0}+\eta_{0}M(t)\right)^{2}.
\end{equation*}
Therefore, we get
\begin{equation*}
\begin{split}
\Psi_{0}^{N}(t)\leq& e^{-C_{1}t}\Psi_{0}^{N}(0)+C\int_0^te^{-C_{1}(t-\tau)}(\tau+1)^{-\frac{3}{2}}(A_{0}+\eta_{0}M(t))^{2}d\tau\\
\leq& C(t+1)^{-\frac{3}{2}}(A^{2}_{0}+\eta^{2}_{0}(M(t))^{2}).
\end{split}
\end{equation*}
So that
\begin{equation}\label{equ5-17}
(t+1)^{\frac{3}{2}}\left\|(n,\omega,\vartheta)\right\|^{2}_{H^N}\leq C(A^{2}_{0}+\eta^{2}_{0}(M(t))^{2}),
\end{equation}
which implies
\begin{equation}\label{equ5-18}
M(t)^{2}\leq C(A^{2}_{0}+\eta^{2}_{0}(M(t))^{2}).
\end{equation}
Since $\eta_{0}$ is small enough, we have that
\begin{equation}\label{equ5-19}
M(t)\leq CA_{0}.
\end{equation}
Thus,
\begin{equation}\label{equ5-20}
\left\|(n,\omega,\vartheta)\right\|_{H^N}\leq CA_{0}(t+1)^{-\frac{3}{4}}\leq C(t+1)^{-\frac{3}{4}}.
\end{equation}

For the case $k=1$, similar to the proof of \eqref{equ5-20}, we also have
\begin{equation}\label{equ5-21}
\frac{d}{dt}\Psi_{1}^{N}(t)+C_{2}\Psi_{1}^{N}(t)\leq C\left\|\nabla(n^{l},\omega^{l},\vartheta^{l})(t)\right\|^{2}_{L^{2}}.
\end{equation}
By H\"{o}lder's inequality and Lemma \ref{lem5-1}, it yields that
\begin{equation}\label{equ5-22}
\begin{split}
\left\|\nabla(n^{l},\omega^{l},\vartheta^{l})(t)\right\|_{L^{2}}
\leq &C(t+1)^{-\frac{5}{4}}\left\|(n_{0},\omega_{0},\vartheta_{0})(t)\right\|_{L^{1}}+C\int_0^t(t-\tau+1)^{-\frac{5}{4}}\left\|F(t)\right\|_{L^{1}}d\tau\\
&+C\int_0^t(t-\tau+1)^{-\frac{3}{4}}\left\|\nabla F(t)\right\|_{L^{1}}d\tau.
\end{split}
\end{equation}
According to H\"{o}lder's inequality and \eqref{equ5-20}, we get that
\begin{equation}\label{equ5-23}
\begin{split}
\left\| F(t)\right\|_{L^{1}}\leq& C\Big(\left\| n\right\|_{L^{2}}\left\| \nabla\omega\right\|_{L^{2}}
+\left\| \omega\right\|_{L^{2}}\left\| \nabla n\right\|_{L^{2}}
+\left\| \omega\right\|_{L^{2}}\left\| \nabla\omega\right\|_{L^{2}}
\\&+\left\| n\right\|_{L^{2}}\left\| \nabla^{2}\omega\right\|_{L^{2}}
+\left\| n\right\|_{L^{2}}\left\| \nabla n\right\|_{L^{2}}
+\left\| \vartheta\right\|_{L^{2}}\left\| \nabla n\right\|_{L^{2}}
\\&+\left\| \omega\right\|_{L^{2}}\left\| \nabla\vartheta\right\|_{L^{2}}
+\left\| \vartheta\right\|_{L^{2}}\left\| \nabla\omega\right\|_{L^{2}}
+\left\| \nabla\omega\right\|_{L^{2}}\left\| \nabla\omega\right\|_{L^{2}}
\\&+\left\| n\right\|_{L^{2}}\left\| \nabla^{2}\vartheta\right\|_{L^{2}}\Big)\\
\leq& C(t+1)^{-\frac{3}{2}}.
\end{split}
\end{equation}
Similarly, we also obtain
\begin{equation}\label{equ5-24}
\left\|\nabla F(t)\right\|_{L^{1}}\leq C(t+1)^{-\frac{3}{2}}.
\end{equation}
By substituting \eqref{equ5-23} and \eqref{equ5-24} into \eqref{equ5-22}, it implies that
\begin{equation*}
\left\|\nabla(n^{l},\omega^{l},\vartheta^{l})(t)\right\|_{L^{2}}\leq C(t+1)^{-\frac{5}{4}}.
\end{equation*}
By \eqref{equ5-21}, it holds
\begin{equation}\label{equ5-25}
\left\|\nabla(n,\omega,\vartheta)\right\|_{H^{N-1}}\leq C(t+1)^{-\frac{5}{4}}.
\end{equation}
Consequently, the proof of this lemma is completed.
\end{proof}

\noindent \textbf{Step 2}. we will use our previous results to prove the decay of the higher-order norm.

Suppose for $k=m$ $(1\leq m\leq N-2)$, we have the following decay rate
\begin{equation}\label{equ5-26}
\left\|\nabla^{m}(n,\omega,\vartheta)\right\|_{H^{N-m}}\leq C(t+1)^{-\frac{3}{4}-\frac{m}{2}}.
\end{equation}
Multiplying the inequality \eqref{equ3-12} by $(t+1)^{\frac{3}{2}+m+\epsilon_{0}}$, it holds that
\begin{equation*}
\begin{split}
\frac{d}{dt}\Big[(t+1)^{\frac{3}{2}+m+\epsilon_{0}}\Psi_{m}^{N}(t)\Big]+(t+1)^{\frac{3}{2}+m+\epsilon_{0}}\Gamma_{m}^{N}(t)
\leq C(t+1)^{\frac{1}{2}+m+\epsilon_{0}}\Psi_{m}^{N}(t).
\end{split}
\end{equation*}
Integrating the above inequality in time, we deduct that
\begin{equation}\label{equ5-27}
\begin{split}
&(t+1)^{\frac{3}{2}+m+\epsilon_{0}}\Psi_{m}^{N}(t)+\int_0^t(\tau+1)^{\frac{3}{2}+m+\epsilon_{0}}\Gamma_{m}^{N}(\tau)d\tau\\
\leq&\Psi_{m}^{N}(0)+ C\int_0^t(\tau+1)^{\frac{1}{2}+m+\epsilon_{0}}\Psi_{m}^{N}(\tau)d\tau\\
\leq&\Psi_{m}^{N}(0)+ C\int_0^t(\tau+1)^{-1+\epsilon_{0}}d\tau\\
\leq&C(t+1)^{\epsilon_{0}}.
\end{split}
\end{equation}
For $k=m+1$, by multiplying the factor $(t+1)^{\frac{5}{2}+m+\epsilon_{0}}$ to \eqref{equ3-12} and integrating it in time, we have
\begin{equation*}
\begin{split}
&(t+1)^{\frac{5}{2}+m+\epsilon_{0}}\Psi_{m+1}^{N}(t)+\int_0^t(\tau+1)^{\frac{5}{2}+m+\epsilon_{0}}\Gamma_{m+1}^{N}(\tau)d\tau\\
\leq&\Psi_{m+1}^{N}(0)+ C\int_0^t(\tau+1)^{\frac{3}{2}+m+\epsilon_{0}}\Psi_{m+1}^{N}(\tau)d\tau\\
\leq&\Psi_{m+1}^{N}(0)+ C\int_0^t(\tau+1)^{\frac{3}{2}+m+\epsilon_{0}}\Gamma_{m}^{N}(\tau)d\tau\\
\leq&C(t+1)^{\epsilon_{0}}.
\end{split}
\end{equation*}
Therefore, we get
\begin{equation}\label{equ5-28}
\left\|\nabla^{m+1}(n,\omega,\vartheta)\right\|_{H^{N-m-1}}\leq C(t+1)^{-\frac{3}{4}-\frac{m+1}{2}},\quad 1\leq m\leq N-2.
\end{equation}

Finally, let $s=\frac{3}{2}$, similar to the proof of formula \eqref{equ4-10}, the decay estimate of the $N-th$ derivative can be obtained, that is
\begin{equation}\label{equ5-29}
\left\|\nabla^{N}(n,\omega,\vartheta)\right\|_{L^{2}}\leq C(t+1)^{-\frac{3}{4}-\frac{N}{2}}.
\end{equation}

\noindent \textbf{Step 3}. we will establish the lower bound of the decay rate for the global solution. Due to \eqref{equ5-4}, we have
\begin{equation}\label{equ5-30}
\left\|U(t)\right\|_{L^{2}}\geq \left\|G\ast U(0)\right\|_{L^{2}}-\int_0^t\left\|G(t-\tau)\ast F(\tau)\right\|_{L^{2}}d\tau.
\end{equation}
From using H\"{o}lder's inequality, \eqref{equ5-9}, \eqref{equ5-12} and \eqref{equ5-23}, it follows that
\begin{equation}\label{equ5-31}
\begin{split}
&\int_0^t\left\|G(t-\tau)\ast F(\tau)\right\|_{L^{2}}d\tau\\
\leq& C\int_0^t(t-\tau+1)^{-\frac{3}{4}}\left\|F(\tau)\right\|_{L^{1}}+e^{-C(t-\tau)}\left\|F(\tau)\right\|_{L^{2}}d\tau\\
\leq& C\int_0^t(t-\tau+1)^{-\frac{3}{4}}\left\|(n,\omega,\vartheta)\right\|_{H^{1}}\left\|\nabla(n,\omega,\vartheta)\right\|_{H^{1}}d\tau\\
&+C\int_0^te^{-C(t-\tau)}\left\|(n,\omega,\vartheta,\nabla\omega)\right\|_{L^{\infty}}\left\|\nabla(n,\omega,\vartheta)\right\|_{H^{1}}d\tau\\
\leq& C\eta_{0}\int_0^t(t-\tau+1)^{-\frac{3}{4}}(\tau+1)^{-\frac{5}{4}}d\tau\\
\leq& C\eta_{0}(t+1)^{-\frac{3}{4}}.
\end{split}
\end{equation}
Since $\eta_{0}$ is sufficiently small, by substituting \eqref{equ5-11} and \eqref{equ5-31} into \eqref{equ5-30}, we have
\begin{equation}\label{equ5-32}
\left\|U(t)\right\|_{L^{2}}\geq c_{1}(t+1)^{-\frac{3}{4}}-C\eta_{0}(t+1)^{-\frac{3}{4}}\geq c_{2}(t+1)^{-\frac{3}{4}},
\end{equation}
where $c_{2}$ is a positive constant. If $t$ is large enough, then
\begin{equation}\label{equ5-33}
\begin{split}
\left\|\Lambda^{-1}U(t)\right\|_{L^{2}}\leq& \left\|\Lambda^{-1}U^{l}(t)\right\|_{L^{2}}+\left\|\Lambda^{-1}U^{h}(t)\right\|_{L^{2}}\\
\leq& C(t+1)^{-\frac{1}{4}}+\int_0^t(t-\tau+1)^{-\frac{1}{4}}\left\|F(\tau)\right\|_{L^{1}}d\tau+\left\|U^{h}(t)\right\|_{L^{2}}\\
\leq& C(t+1)^{-\frac{1}{4}}.
\end{split}
\end{equation}
According to Lemma \ref{lem2-5}, we obtain
\begin{equation}\label{equ5-34}
\left\|U(t)\right\|_{L^{2}}\leq C\left\|\Lambda^{-1}U(t)\right\|^{\frac{k}{k+1}}_{L^{2}}\left\|\nabla^{k}U(t)\right\|^{\frac{1}{k+1}}_{L^{2}},
\end{equation}
which implies that
\begin{equation}\label{equ5-35}
\left\|\nabla^{k}U(t)\right\|_{L^{2}}\geq c_{\ast}(t+1)^{-\frac{3}{4}-\frac{k}{2}},
\end{equation}
where $c_{\ast}$ is a positive constant. So that, we get
\begin{equation}\label{equ5-36}
\left\|\nabla^{k}(\rho-\rho_{\ast},u,\theta-\theta_{\ast})\right\|_{L^{2}}\geq c_{\ast}(t+1)^{-\frac{3}{4}-\frac{k}{2}}.
\end{equation}
Therefore, we complete the proof of Theorem \ref{the1-3}.  \hfill $\Box$

\section*{Data availability}
Data sharing is not applicable to this article as no datasets were generated or analysed during the current study.

\section*{Conflict of interest} The authors declare that there is no conflict of interest. We also declare that this manuscript has no associated data.

\section*{Acknowledgements}
We would like to express our gratitude to Dr. Tang Houzhi from Capital Normal University for his answers to our confusion, as well as for his feedback and suggestions on the article, which have helped us to improve our manuscript. This work was supported by the National Natural Science Foundation of China (No. 12071098) and the Fundamental Research Funds for the Central Universities (No. 2022FRFK060022).

\end{document}